\documentclass[reqno, 10pt]{amsart}

\usepackage{amsmath,amsthm,amssymb,amsfonts, amscd, bm, graphicx, mathrsfs}

\usepackage{youngtab}

\usepackage[cmtip,all]{xy}
\usepackage[top=2cm, bottom=2cm, left=3.3cm, right=3.3cm]{geometry}

\usepackage{mathrsfs}

\usepackage[pagebackref=true, colorlinks]{hyperref}

\hypersetup{colorlinks=true,citecolor=blue,linkcolor=blue,urlcolor=blue,
pdfstartview=FitH}
\newtheorem{thm}{Theorem}[section]
\newtheorem{lem}[thm]{Lemma}
\newtheorem{prop}[thm]{Proposition}

\newtheorem{conjecture}[thm]{Conjecture}
\theoremstyle{definition}

\theoremstyle{remark}
\newtheorem{rem}[thm]{Remark}

\numberwithin{equation}{section}

\newcommand{\bc}{{\bf c}}

\newcommand{\R}{\mathbb R}

\newcommand{\cE}{\mathcal E}

\newcommand{\cV}{\mathcal V}
\newcommand{\cH}{\mathcal H}
\newcommand{\cD}{\mathcal D}
\newcommand{\cO}{\mathcal O}

\allowdisplaybreaks  

\makeatletter
\DeclareRobustCommand\widecheck[1]{{\mathpalette\@widecheck{#1}}}
\def\@widecheck#1#2{%
    \setbox\z@\hbox{\m@th$#1#2$}%
    \setbox\tw@\hbox{\m@th$#1%
       \widehat{%
          \vrule\@width\z@\@height\ht\z@
          \vrule\@height\z@\@width\wd\z@}$}%
    \dp\tw@-\ht\z@
    \@tempdima\ht\z@ \advance\@tempdima2\ht\tw@ \divide\@tempdima\thr@@
    \setbox\tw@\hbox{%
       \raise\@tempdima\hbox{\scalebox{1}[-1]{\lower\@tempdima\box
\tw@}}}%
    {\ooalign{\box\tw@ \cr \box\z@}}}
\makeatother

\begin{document}

\title[The fourth moment of truncated Eisenstein series]{The fourth moment of truncated Eisenstein series}


\author{Goran Djankovi{\' c}}

\address{
\begin{flushleft}
University of Belgrade,  
Faculty of Mathematics, \\
Studentski Trg 16, p.p. 550, 
11000 Belgrade, Serbia
\end{flushleft}
}

\email{goran.djankovic@matf.bg.ac.rs}

\author{Rizwanur Khan}

\address{
\begin{flushleft}
 University of Texas at Dallas, Department of Mathematical Sciences \\ Richardson, TX 75080-3021, USA
\end{flushleft}
}

\email{rizwanur.khan@utdallas.edu}


\begin{abstract}
We obtain an asymptotic for the fourth moment of truncated Eisenstein series of large Laplacian eigenvalue, verifying for the first time that the main term corresponds to Gaussian random behavior. This is a manifestation of the Random Wave Conjecture, which for Eisenstein series was formulated by Hejhal and Rackner over thirty years ago. Our innovation is to tackle the problem after introducing, at no cost, an extra averaging over the truncation parameter.
 \end{abstract}



\thanks{The second author was supported by the National Science Foundation grants DMS-2344044 and DMS-2341239.}

\maketitle


\section{Introduction}

In arithmetic settings, good progress has been made in recent years on the Quantum Unique Ergodicity (QUE) problem. However towards the more general Random Wave Conjecture (RWC), there has been more limited success. The RWC refers to the principle of Berry \cite{Be} that highly excited eigenfunctions of a classically ergodic system should manifest Gaussian random behavior. Hejhal and Rackner \cite{HR} and Hejhal and Str\"{o}mbergsson \cite{HS} have supported this conjecture numerically for Laplacian eigenfunctions on the modular surface $\Gamma_0(N)\backslash\mathbb{H}$. Further, in \cite[section 7.3]{HR}, the conjecture was formulated for the continuous spectrum  of the Laplacian, the Eisenstein series, even though they are not $L^2$-integrable. 

We focus on the modular surface, where one can consider joint eigenfunctions of the Laplacian and Hecke operators. A desciption of the RWC in terms of moments of even or odd Hecke-Maass cusp forms $f$ for $\Gamma=SL_2(\mathbb{Z})$ with Laplacian eigenvalue $\frac14+t_f^2$ can be given as follows. For any integer $p\ge 1$ and any fixed, compact regular set $\Omega\subset\Gamma\backslash\mathbb{H} $, it is expected that
\begin{align*}
 \frac{1}{\mu(\Omega)}  \int_{\Omega} \left( \mu( \Gamma \backslash\mathbb{H} )^{\frac12}  \frac{f(z) }{   \| f \|_2} \right)^p d\mu z \sim c_p
\end{align*}
as $f$ traverses any sequence with $t_f\to\infty$, where we use the usual hyperbolic measure, $\|\cdot \|_p$ denotes the $L^p$-norm, and $c_p$ is the $p$-th moment of a standard normal random variable. For the Eisenstein series, the conjecture is that
\begin{align}
\label{continuous} \frac{1}{\mu(\Omega)}  \int_{\Omega} \left( \mu( \Gamma \backslash\mathbb{H} )^{\frac12}  \frac{\tilde{E}(z,\frac12+iT)}{ \sqrt{2\log T}}\right)^p d\mu z \sim c_p
\end{align}
as $T\to \infty$, where $\tilde{E}(z,s)$ can be taken to be either $\frac{\xi(1+2iT)}{|\xi(1+2iT)|} E(z,\frac12+iT)$ or $\frac{\xi(1+2iT)}{|\xi(1+2iT)|} E_A(z,\frac12+iT)$. The notation will be defined below but note that $E_A(z,s)$ is the truncated Eisenstein series, which equals $E(z,s)$ on the fundamental domain but with the constant term of its Fourier series subtracted off for $\Im(z)>A$. Multiplication by $\frac{\xi(1+2iT)}{|\xi(1+2iT)|}$ serves to ensure that $\tilde{E}(z,\frac12+iT)$ is real valued. And finally, note that $\sqrt{2 \log T}$ roughly equals $\|E_A(\cdot, \frac{1}{2} + iT)\|_2$. One can expand these conjectures to include the set $\Omega=  \Gamma\backslash\mathbb{H}$, which is not compact, as long as $p$ is not so large that the moments diverge (in the limit). This situation is likely easier than the case of an arbitrary compact set $\Omega$. Note that for \eqref{continuous}, when working with $\Omega=  \Gamma\backslash\mathbb{H}$  one must use the definition of $\tilde{E}(z,s)$ in terms of $E_A(z,s)$ instead of $E(z,s)$ since the latter is not square-integrable. 

The first nontrivial case of the RWC is $p=2$, better known as QUE. This was established for the Eisenstein series by Luo and Sarnak \cite{LS} well before the Hecke-Maass cusp forms by Lindenstrauss \cite{Li} and Soundararajan \cite{So}. The case $p=3$ has no main term (since $c_p=0$ for odd $p$) and was settled for both Eisenstein series and Hecke-Maass cusp forms, by Watson \cite{W} in the case $\Omega=  \Gamma \backslash\mathbb{H}$ even before QUE was resolved, and recently by Huang \cite{Hu} in the case of compact $\Omega$. This brings us to $p=4$, which is the largest value for which the RWC has any sort of resolution. For dihedral Maass newforms, the case $p=4$ and $\Omega=\Gamma_0(N)\backslash\mathbb{H}$ was proven by Humphries and the second author \cite{HK}. Until now, this was the only resolved fourth moment instance of the RWC, a fact that seems surprising since one might have naively expected the problem for Eisenstein series to be easier and to be solved first. However this is not the case as the fourth moment of Eisenstein series has a different set of challenges, which in this paper we are able to overcome. For general Hecke-Maass cusp forms, the fourth moment is only known conditionally on the Generalized Lindel\"{o}f Hypothesis by the work of Buttcane and the second author \cite{BK}. 

We state again Hejhal and Rackner's conjecture for the fourth moment of Eisenstein series. The following is equivalent to \eqref{continuous} with $p=4$, for which $c_4=3$, and $\Omega=  \Gamma \backslash\mathbb{H}$.

\begin{conjecture} \label{the-conjecture} \cite{HR} Fix a constant $A> 1$.  We have
$$
\| E_A( \cdot,  \tfrac{1}{2} +iT  ) \|_4^4   \sim  \frac{36}{\pi}\log^2 T
$$
as $T \rightarrow \infty$.
\end{conjecture}
\noindent Until now, the best result directly towards this conjecture was the upper bound 
\begin{align}
\label{EA-bound} \| E_A( \cdot,  \tfrac{1}{2} +iT  ) \|_4^4\ll \log^2 T
\end{align}
 due to Spinu \cite{S} and Humphries \cite{H}. There is also an alternative version of the above conjecture, using Zagier's \cite{Z} regularized integral, which was formulated and solved by the authors in \cite{DK1, DK2}. However, the regularized version is a weaker result in the sense that it is implied by but does not imply Conjecture \ref{the-conjecture}. Nevertheless, this paper builds upon \cite{DK2}.
   
In Conjecture \ref{the-conjecture} and all previous works on it, the truncation parameter $A$ is a fixed constant. Our main novelty is to give a more fluid role to $A$. We will introduce an extra averaging over $A$ and establish Conjecture \ref{the-conjecture} on average. This may seem like we are proving something weaker, but we will show that this averaged result actually returns the conjecture for fixed $A$. In this way, we will fully resolve Conjecture \ref{the-conjecture}.
  
 \begin{thm} \label{Thm:main}  Fix a constant $A> 1$.  We have
$$
\| E_A( \cdot,  \tfrac{1}{2} +iT  ) \|_4^4   \sim  \frac{36}{\pi}\log^2 T
$$
as $T \rightarrow \infty$.
\end{thm}

\section{Discussion of the proof}

Part of the difficulty is the unique set-up of the problem in the Eisenstein series case due to convergence issues. If we ignore for a moment all convergence issues then the usual starting point would be to use Parseval's identity and spectral decomposition to express the fourth moment as 
\[
\langle E^2(\cdot ,\tfrac12+iT) , E^2(\cdot ,\tfrac12+iT) \rangle \stackrel{.}{=} \sum_{j\ge 1} |\langle E^2(\cdot ,\tfrac12+iT), u_j \rangle|^2,
\]
where the sum is over an orthonormal basis of even or odd Hecke-Maass cusp forms, and we use the symbol $\stackrel{.}{=}$  to mean that the equality is only true in spirit (for example, in this case to have a true equality we would need to also write down the contribution from the rest of the spectrum). Next by `unfolding' we have
\begin{align}
\label{unfolded} \sum_{j\ge 1} |\langle E^2(\cdot ,\tfrac12+iT), u_j \rangle|^2 \stackrel{.}{=} \sum_{j\ge 1} \frac{L(\tfrac12, u_j)^2 |L(\tfrac12+2iT, u_j)|^2 }{|\zeta(1+2iT)|^4}|\mathcal{H}(t_j)|^2,
\end{align}
where $t_j$ is the spectral parameter of $u_j$ and $\mathcal{H}(t_j)$ is a ratio of gamma functions which dictates that the sum is essentially supported on $|t_j|<2T$. At this point the problem is reduced to that of evaluating an average of central $L$-values, which was accomplished by the authors in \cite{DK2}. Before we continue the discussion, recall that for $F(z)$ automorphic with rapid decay at the cusp, the technique of unfolding refers to the manipulation
\begin{align*}
 &\int_{\Gamma  \backslash \mathbb{H}} E(z,s) F(z) d\mu z= \int_{\Gamma  \backslash \mathbb{H}} \Big(  \frac12 \sum_{\gamma\in\Gamma_\infty\backslash \Gamma} \Im(\gamma z)^s\Big) F(z) d\mu z = \frac12 \int_{\Gamma_\infty  \backslash \mathbb{H}} \Im( z)^s F(z) d\mu z,
\end{align*}
where the final integral, over the rectangle $[0,1]\times[0,\infty]$, may be evaluated by Mellin transformation after taking a Fourier series expansion of $F(z)$.

Unfortunately the above strategy does not actually work because the starting object $\langle E^2(\cdot ,\tfrac12+iT) , E^2(\cdot ,\tfrac12+iT) \rangle$ is not convergent. If we replace $E^2(\cdot ,\tfrac12+iT)$ with $E_A^2(\cdot ,\tfrac12+iT)$, then we do have a convergent object. Then after applying Parseval's identity, we want to pass from $\langle E_A^2(\cdot, \tfrac{1}{2} +iT), u_j\rangle$ to $\langle E^2(\cdot, \tfrac{1}{2} +iT), u_j\rangle$ since we can evaluate \eqref{unfolded}. We get (with details to follow) that
\begin{multline*}
 \sum_{j \ge 1} |\langle E_A^2(\cdot, \tfrac{1}{2} +iT), u_j\rangle|^2  \stackrel{.}{=} \sum_{j \ge 1} \left|  \langle E^2( \cdot, \tfrac{1}{2} + iT), u_j \rangle \right|^2+  \sum_{j \ge 1} \left| \langle  y^{\frac12+iT}E_A(z,\tfrac12+iT) , u_j   \rangle_A \right|^2\\
   - 2 \Re \left(   \sum_{j \ge 1} \overline{\langle E^2( \cdot, \tfrac{1}{2} + iT), u_j \rangle }   \langle  y^{\frac12+iT}E_A(z,\tfrac12+iT) , u_j   \rangle_A  \right),
\end{multline*}
where $\langle \cdot , \cdot  \rangle_A$ denotes an inner product only over the region of the fundamental domain with $\Im(z)>A$. This region is already a rectangle, so we can skip the unfolding process for these inner products. Consider the third sum on the right hand side, for which we find
\begin{align}
\nonumber &\sum_{j\ge 1}\overline{\langle E^2( \cdot, \tfrac{1}{2} + iT), u_j \rangle }   \langle  y^{\frac12+iT}E_A(z,\tfrac12+iT) , u_j   \rangle_A \\
\label{unfolded-2} &\stackrel{.}{=} \sum_{j\ge 1} \frac{L(\tfrac12, u_j) L(\tfrac12-2iT, u_j)}{\zeta^2(1-2iT)} \overline{\mathcal{H}(t_j)} \frac{1}{2\pi i} \int_{(3)} \frac{L(\tfrac12+s, u_j) L(\tfrac12+s+2iT, u_j)}{\zeta(1+2iT)\zeta(1+2s+2iT)} \mathcal{H}(s,t_j) A^{-s} \frac{ds}{s},
\end{align}
where $\mathcal{H}(s,t_j)$ is a ratio of gamma functions that equals $\mathcal{H}(t_j)$ for $s=0$. It is interesting that the problem boils down to a mean value of $L$-functions that resembles \eqref{unfolded}, but one that is more difficult. If we shift contours left then we can recover \eqref{unfolded} from the pole at $s=0$, but it is not clear what to do with the shifted integral. In any case we cannot shift far without the Riemann Hypothesis due to the $\zeta(1+2s+2iT)$ factor in the denominator and the fact that the function $\mathcal{H}(s,t_j)$ does not preclude $|s|$ from being of size $T$. Thus essentially the difference between the mean values \eqref{unfolded} and \eqref{unfolded-2} is that the former involves only central values while in the latter we must treat the $L$-functions along the critical line (or close to it), which seems very difficult to do.

We are able to break through with a new approach. We `smoothen out' the sharp truncation of $E_A(z,\tfrac12+iT)$ by introducing in \eqref{unfolded-2} an integral with respect to $A$ over a shrinking interval. We show that to prove our main theorem, it suffices to prove this smooth version. The introduction of the $A$-integral enables us to use integration by parts with respect to $A$ to restrict to $|s|<T^\alpha$ for $\alpha>0$ which we can fix as small as we like. Thus the problem becomes much closer to \eqref{unfolded}.

We use an approximate functional equation to get a handle on the central values $L(\tfrac12, u_j) L(\tfrac12-2iT, u_j)$. But for $L(\tfrac12+s, u_j) L(\tfrac12+s+2iT, u_j) $ we can already expand this into a Dirichlet series. Then using properties of $\mathcal{H}(s,t_j)$ we find surprisingly that this Dirichlet series is actually shorter than that of the approximate functional equation of $L(\tfrac12, u_j) L(\tfrac12-2iT, u_j)$. Thus since \eqref{unfolded} can be evaluated, it is plausible we should be able to evaluate \eqref{unfolded-2} too. This is precisely what we do, using Kuznetsov's formula. Our success in treating the off-diagonal is due to the observation of the short Dirichlet series together with the smoothing device which enforces $|s|<T^\alpha$.  Altogether we find several main terms from various pieces, which nicely combine to give the conjectured main term of $ \frac{36}{\pi}\log^2 T$.

It is also worth remarking that at some point in the argument (see section \ref{Section_diag_II}), we must make use of a sub-Weyl strength subconvexity bound for the Riemann Zeta function. We find it quite surprising that we need such a deep result, which was not required in the treatment of \eqref{unfolded} given in \cite{DK2}, nor in the fourth moment of dihedral Maass newforms established in \cite{HK}. This input may be indicative of the difficulty of Conjecture \ref{the-conjecture}.

\section{Notation and preliminaries}

The constant term of Eisenstein series $E(z, \frac{1}{2} +iT)$ is given by
$$
e(y, \tfrac{1}{2} + i T)=y^{\frac{1}{2} + i T} + y^{\frac{1}{2} - i T}  \frac{\xi(1-2 i T)}{\xi(1+ 2 i T)}=y^{\frac{1}{2} + i T} + \bc \cdot y^{\frac{1}{2} - i T} ,
$$
where we denote
$$
\bc:=\bc(T)= \frac{\xi(1-2 i T)}{\xi(1+ 2 i T)},
$$
and $\xi(s)=\pi^{-\frac{s}{2}} \Gamma(\tfrac{s}{2})\zeta(s)$ is the completed Riemann zeta function.

Let $\mathcal{F}$ denote the usual fundamental domain for $SL_2(\mathbb{Z})  \backslash \mathbb{H}$, and for $A > 1$, let
$$
\mathcal{F}_A=\{ z \in \mathcal{F}  : \Im(z) \le A \}  ,   \qquad   \mathcal{C}_A=\{ z \in \mathcal{F}  : \Im(z) > A \}
$$
denote the compact and the cuspidal part of the fundamental domain, respectively, which lie below and above the line $\Im(z)=A$, respectively. 
Then
the truncated Eisenstein series is defined (following \cite{S}) by
$$
E_A(z, s)=\begin{cases}  E(z,s),  & \text{if  } z \in \mathcal{F}_A \\ E(z, s)- e(y, s),  & \text{if  } z \in \mathcal{C}_A.
\end{cases}
$$
This definition may be extended to $\mathbb{H}$ by $SL_2(\mathbb{Z})$ translates, but anyway when taking the fourth moment we integrate only over $\mathcal{F}$.

For $z \in \mathcal{C}_A$ we have the Fourier expansion
\begin{equation}  \label{eq:Fourier_E_A}
E_A(z, \tfrac{1}{2}+iT)=\frac{2}{\xi(1+ 2 iT)} \sum_{n \neq 0} \tau_{}(|n|, T) \sqrt{y} K_{iT}(2 \pi |n| y ) e(nx),
\end{equation}
where 
we denote for   positive integers $m$ and $\gamma \in \R$
$$
\tau(m, \gamma)=\sum_{ab=m}  \left( \frac{a}{b} \right)^{i \gamma}=m^{-i \gamma} \sum_{a|m} a^{2i \gamma}.
$$
This generalized divisor function
is real valued, $\tau(m, \gamma)=\tau(m, -\gamma)$, and $|\tau(m, \gamma)|\le \tau(m)$. We will also be using that $E_A(z, \tfrac{1}{2}+iT) \xi(1+ 2 iT)$ is real valued, a fact that was noted in \cite[equation (5.2)]{DK1}.


\bigskip

Let $\{ u_j \; : j \ge 1 \}$ denote an orthonormal basis of even or odd Hecke-Maass cusp forms for the modular group $\Gamma=\mathrm{SL}_2(\mathbb{Z})$. Let $\frac{1}{4}+t_j^2$, with $t_j >0$ denote the corresponding Laplacian eigenvalues and
let $\lambda_j(m)$  denote the (real) eigenvalues of the $m$-th Hecke operator corresponding to $u_j$:  $T_m u_j =\lambda_j(m) u_j$  for all $m\ge 1$.
Moreover we write $\lambda_j(-m) = \lambda_j(m)$ for $u_j$ even and $\lambda_j(-m) =
-\lambda_j(m)$ for $u_j$ odd.  Each such Hecke-Maass cusp form $u_j$ has
the Fourier expansion
$$
u_j(z) 
=
\rho_j(1)  \sum_{m \neq 0} \lambda_j(m)  \sqrt{y}  K_{i t_j}(2 \pi |m| y) e(mx),
$$
where we have the following
formula relating the normalizing factor $\rho_j(1)$ with the symmetric square $L$-function associated to $u_j$:
\begin{equation} \label{rho_1}
|\rho_j(1)|^2=\frac{2\cosh(\pi t_j)}{L(1, \text{sym}^2 u_j)}=\frac{2 \pi}{\Gamma(\tfrac{1}{2}+i t_j)\Gamma(\tfrac{1}{2}-i t_j)  L(1, \text{sym}^2 u_j)}.
\end{equation}
In both cases, even or odd, $u_j$ is real-valued on $\mathbb{H}$.


The $L$-function attached to the Hecke-Maass cusp form $u_j$ is defined 
for $\Re(s)>1$ by
$$
L(s, u_j) =\sum_{m \ge 1}   \frac{\lambda_j(m)}{m^s}.
$$

The completed $L$-function for \emph{even} $u_j$ is
$$
\Lambda(s, u_j)=\pi^{-s} \Gamma\left(\frac{s+ it_j}{2} \right) \Gamma\left(\frac{s- it_j}{2} \right) L(s, u_j)
$$
and it satisfies the functional equation
$\Lambda(s, u_j)= \Lambda(1- s, u_j)$.
For \emph{odd} $u_j$, the completed $L$-function is
$$
\Lambda(s, u_j)=\pi^{-1-s} \Gamma\left( \frac{1+ s+ i t_j}{2} \right)\Gamma\left( \frac{1+ s- i t_j}{2} \right)L(s, u_j)
$$
and satisfies $\Lambda(s, u_j)= - \Lambda(1- s, u_j)$.



\subsection{Stirling's approximation}  Let $\delta>0$ be fixed. For $z\in\mathbb{C}$ with $\Re(z)>\delta$ and $t\in\mathbb{R}$ with  $|t|>2|z+1|^2$, we will frequently appeal to Stirling's approximation
\begin{align}
\label{eq:stirling} \Gamma(z+it)=  \sqrt{2\pi} |t|^{(z+it -\frac12)} \exp\left( -\frac{\pi}{2}|t|   -it  + i   \mathrm{sgn}(t)\frac{\pi}{2} \left(z-\frac12\right) \right) \left(1+ O\left( \frac{|z+1|^2}{|t|}\right)\right),
\end{align}
which can be refined to arbitrary precision by replacing the factor $1+ O\Big( \frac{|z+1|^2}{|t|}\Big)$ with 
\begin{align}
\label{stirling2} 1+ \sum_{k=1}^{N} \frac{c_{2k}(z)}{t^k} +  O_N\bigg(\left( \frac{|z+1|^2}{|t|}\right)^{N+1}\bigg),
\end{align}
where $c_{2k}(z)$ are polynomials in $z$ of degree at most $2k$ and $N\ge 1$ is any integer.

To see the statements above, recall that Stirling's approximation gives
\begin{align}
\label{log-gamma} \log \Gamma(z+it) = (z+it-\tfrac12)\log (z+it) -(z+it) + \sum_{j=1}^{M} \frac{c_j}{(it)^j(1+\frac{z}{it})^j}+O_M\left(\frac{1}{|z+it|^{M+1}}\right)
\end{align} 
for some constants $c_j$ and any integer $M\ge 1$. We write
\begin{align}
\label{log} \log (z+it) = \log(it)+\log \Big(1+\frac{z}{it}\Big) = \log|t| +i \frac{\pi}{2} \mathrm{sgn}(t) + \frac{z}{it} - \sum_{j=2}^\infty \frac{(-1)^{j}}{j}\Big(\frac{z}{it}\Big)^j.
\end{align}
Now inserting \eqref{log} into \eqref{log-gamma}, taking a power series expansion of $(1+\frac{z}{it})^{-j}$, taking $M$ as large as we like, and finally taking the exponential of both sides of \eqref{log-gamma}, gives \eqref{eq:stirling} and \eqref{stirling2}.


\section{A smooth version}

We will prove a version of Theorem \ref{Thm:main} with an extra averaging over $A$. One may interpret the effect of this averaging as a way to smoothen out the sharp truncation of $E_A(z,\frac12+iT)$.
\begin{prop}\label{the-prop} Fix $B>1$ and $0<\alpha< \frac{1}{100}$. Let $h(A)$ be a smooth, non-negative real function supported on $B-T^{-\frac{\alpha}{2}}< A < B+T^{-\frac{\alpha}{2}}$ for $T>0$, satisfying
\begin{align}
\nonumber &T^{-\frac{\alpha}{2}} \ll \hat{h}(0) \ll T^{-\frac{\alpha}{2}},\\
\label{h-derivatives}   &h^{(k)}(A)\ll_k \left(T^{\frac{\alpha}{2}}\right)^k,
\end{align}
where $ \hat{h}(0)=\int_{-\infty}^\infty h(A) dA$ and $k$ is any nonnegative integer. 
Then for $\alpha$ small enough, we have
$$
\int_{-\infty}^\infty h(A) \| E_A( \cdot,  \tfrac{1}{2} +iT  ) \|_4^4   \ dA \;  \;   \sim \;  \hat{h}(0) \frac{36}{\pi}\log^2 T
$$
as $T \rightarrow \infty$.
\end{prop}

\

The rest of the paper will focus on proving Proposition \ref{the-prop}, because that is enough to imply our main theorem.
\begin{lem}
Proposition \ref{the-prop} implies Theorem \ref{Thm:main}.
\end{lem} 
\begin{proof}
Suppose that Proposition \ref{the-prop} is true. We need to prove that 
\[
\hat{h}(0)  \| E_B( \cdot,  \tfrac{1}{2} +iT  ) \|_4^4   \sim  \hat{h}(0) \frac{36}{\pi}\log^2 T.
\]
The left hand side equals
\begin{align*}
 \int_{-\infty}^\infty h(A) \  \| E_B( \cdot,  \tfrac{1}{2} +iT  ) \|_4^4 \ dA=  \int_{-\infty}^\infty h(A) \int\limits_{\mathcal{F}}  \ |E_B( z,  \tfrac{1}{2} +iT  )|^4 d\mu z \ dA=I_1+I_2+I_3,
\end{align*}
where $d\mu(z)=\frac{dx dy}{y^2}$ for $z=x+i y$, and
\begin{align*}
&I_1:= \int_{-\infty}^\infty h(A) \int\limits_{\substack{z\in\mathcal{F} \\ \Im(z)< B-T^{-\frac{\alpha}{2}}}}   \left|E_A( z,  \tfrac{1}{2} +iT  ) + (E_B( z,  \tfrac{1}{2} +iT  )-E_A( z,  \tfrac{1}{2} +iT  ))\right|^4 d\mu z \  dA, \\
&I_2:= \int_{-\infty}^\infty h(A) \int\limits_{\substack{z\in\mathcal{F} \\ B-T^{-\frac{\alpha}{2}} \le \Im(z) \le B+T^{-\frac{\alpha}{2}} }}  \left|E_A( z,  \tfrac{1}{2} +iT  ) + (E_B( z,  \tfrac{1}{2} +iT  )-E_A( z,  \tfrac{1}{2} +iT  ))\right|^4 d\mu z \  dA, \\
&I_3:= \int_{-\infty}^\infty h(A) \int\limits_{\substack{z\in\mathcal{F} \\ \Im(z)> B+T^{-\frac{\alpha}{2}}}}    \left|E_A( z,  \tfrac{1}{2} +iT  ) + (E_B( z,  \tfrac{1}{2} +iT  )-E_A( z,  \tfrac{1}{2} +iT  ))\right|^4 d\mu z \  dA.
\end{align*}
Since $A$ is restricted to $B-T^{-\frac{\alpha}{2}}< A < B+T^{-\frac{\alpha}{2}}$ by the support of $h(A)$, we have 
\[
E_B( z,  \tfrac{1}{2} +iT  )-E_A( z,  \tfrac{1}{2} +iT  ) = 0
\]
in $I_1$ and $I_3$. In $I_2$, we have
\[
|E_B( z,  \tfrac{1}{2} +iT  )-E_A( z,  \tfrac{1}{2} +iT  ) |\ll y^{\frac12}\ll 1.
\]
Thus
\begin{align}
\label{I2-error} I_2 &= \int_{-\infty}^\infty h(A) \int\limits_{\substack{z\in\mathcal{F} \\ B-T^{-\frac{\alpha}{2}} \le \Im(z) \le B+T^{-\frac{\alpha}{2}}}} |E_A( z,  \tfrac{1}{2} +iT  )|^4 d\mu z \  dA \\
\nonumber &+O\Bigg( \int_{-\infty}^\infty h(A)\int\limits_{\substack{z\in\mathcal{F} \\ B-T^{-\frac{\alpha}{2}} \le \Im(z) \le B+T^{-\frac{\alpha}{2}} }} 1 \ d\mu z \  dA\Bigg)\\
 \nonumber &+O\Bigg( \int_{-\infty}^\infty h(A)\int\limits_{\substack{z\in\mathcal{F} \\ B-T^{-\frac{\alpha}{2}} \le \Im(z) \le B+T^{-\frac{\alpha}{2}} }} |E_A( z,  \tfrac{1}{2} +iT  )|^3 d\mu z \  dA\Bigg).
\end{align}
We have shown that the main term of $I_1+I_2+I_3$ equals
\[
\int_{-\infty}^\infty h(A) \int_\mathcal{F} |E_A( z,  \tfrac{1}{2} +iT  ))|^4 d\mu z   \ dA \sim   \hat{h}(0) \frac{36}{\pi}\log^2 T,
\]
since we are assuming Proposition \ref{the-prop}. Thus is remains to show that the error term in \eqref{I2-error} is $o(\hat{h}(0)  \log^2 T)$. We will show that it is actually bounded by $\hat{h}(0)$ times a negative power of $T$.

Write $\Omega=\{z\in\mathcal{F}: B-T^{-\frac{\alpha}{2}} \le \Im(z) \le B+T^{-\frac{\alpha}{2}}  \}$ and note that $\mu(\Omega) \ll T^{-\frac{\alpha}{2}}$. The first error term in \eqref{I2-error} is bounded by $\hat{h}(0) T^{-\frac{\alpha}{2}}$. By H\"{o}lder's inequality and \eqref{EA-bound}, the second error term is bounded by
\begin{align*}
\int_{-\infty}^\infty h(A)  \Bigg( \int\limits_\Omega 1 \ d\mu z\Bigg)^\frac14  \Bigg( \int\limits_\mathcal{F} |E_A( z,  \tfrac{1}{2} +iT  )|^4 d\mu z\Bigg)^\frac34 dA
\ll \hat{h}(0) (T^{-\frac{\alpha}{2}} )^\frac14 (\log^2  T)^\frac34 \ll \hat{h}(0) T^{-\frac{\alpha}{9}}.
\end{align*}
\end{proof}


\section{Spectral decomposition of the fourth moment}

Let $\langle f, g \rangle=\int_{\mathcal{F}} f(z) \overline{g(z)} d\mu(z)$ denote Petersson's inner product. The starting point is  Parseval's identity
$$
\| E_A(\cdot, \tfrac{1}{2} + iT) \|_4^4=\int_{\mathcal{F}} |E_A(z, \tfrac{1}{2} + iT)|^4 d \mu z
$$
$$
=     |\langle E_A^2(\cdot, \tfrac{1}{2} +iT), (\tfrac{3}{\pi})^{1/2}\rangle|^2 
+ \sum_{j \ge 1} |\langle E_A^2(\cdot, \tfrac{1}{2} +iT), u_j\rangle|^2 
+\frac{1}{4 \pi} \int_{\mathbb{R}} |\langle E_A^2(\cdot, \tfrac{1}{2} +iT), E(\cdot, \tfrac{1}{2} +it)\rangle|^2 dt.
$$
For   a Hecke-Maass cusp form $u_j(z)$  we have
$$
 \langle E_A^2(\cdot,  \tfrac{1}{2} + iT), u_j \rangle
=\langle E^2( \cdot, \tfrac{1}{2} + iT), u_j \rangle
- \int_{\mathcal{C}_A}  2 e(y, \tfrac{1}{2} + iT) E_A(z, \tfrac{1}{2} + iT)  u_j(z) d \mu z.
$$
Thus if we denote
$$
H_A(z)=
\begin{cases}
0, &  \text{if}  \;  z \in \mathcal{F}_A, \\
2  e(y, \tfrac{1}{2} + iT) E_A(z, \tfrac{1}{2} + iT), &  \text{if}  \;  z \in \mathcal{C}_A , \\
\end{cases}
$$
we get
\begin{align*}
& \sum_{j \ge 1} |\langle E_A^2(\cdot, \tfrac{1}{2} +iT), u_j\rangle|^2
 = \sum_{j \ge 1} \left|  
 \langle E^2( \cdot, \tfrac{1}{2} + iT), u_j \rangle
-   \langle  H_A, u_j   \rangle 
  \right|^2    \\
&   = \sum_{j \ge 1} \left|  
 \langle E^2( \cdot, \tfrac{1}{2} + iT), u_j \rangle \right|^2
+  \sum_{j \ge 1} \left| \langle  H_A, u_j   \rangle 
  \right|^2
  - 2 \Re \left(   \sum_{j \ge 1} \overline{\langle E^2( \cdot, \tfrac{1}{2} + iT), u_j \rangle }   \langle  H_A, u_j   \rangle  \right).
\end{align*}
We rewrite the middle sum by employing Parseval's identity again:
$$
\sum_{j \ge 1} \left| \langle  H_A, u_j   \rangle 
  \right|^2= \langle H_A, H_A \rangle  - \left| \langle  H_A, (\tfrac{3}{\pi})^{1/2}   \rangle 
  \right|^2   -\frac{1}{4 \pi}  \int_{\R}  \left| \langle  H_A, E(\cdot, \tfrac{1}{2} + it)   \rangle 
  \right|^2  dt.
$$
From the Fourier expansion (\ref{eq:Fourier_E_A}) we see that the function $H_A(z)$ is orthogonal to constants, and thus we arrive at the following expression:

\begin{lem} \label{lema:spectral_expansion}
For any $A>1$, we have
\begin{align*}
\| E_A(\cdot, \tfrac{1}{2} + iT) \|_4^4
=  
&   |\langle E_A^2(\cdot, \tfrac{1}{2} +iT), (\tfrac{3}{\pi})^{1/2}\rangle|^2  + \sum_{j \ge 1} \left|  
 \langle E^2( \cdot, \tfrac{1}{2} + iT), u_j \rangle \right|^2   \\
 & + \langle H_A, H_A \rangle 
  - 2 \Re \left(   \sum_{j \ge 1} \overline{\langle E^2( \cdot, \tfrac{1}{2} + iT), u_j \rangle }   \langle  H_A, u_j   \rangle  \right)  \\
  &
 +\frac{1}{4 \pi} \int_{\mathbb{R}} |\langle E_A^2(\cdot, \tfrac{1}{2} +iT), E(\cdot, \tfrac{1}{2} +it)\rangle|^2 dt 
     -\frac{1}{4 \pi}  \int_{\R}  \left| \langle  H_A, E(\cdot, \tfrac{1}{2} + it)   \rangle 
  \right|^2  dt. 
\end{align*}
\end{lem}
We denote the `cross-term' in the parentheses in the second line by
$$
\Xi(A):=
 \sum_{j \ge 1} \overline{\langle E^2( \cdot, \tfrac{1}{2} + iT), u_j \rangle }   \langle  H_A, u_j   \rangle.
$$

For the first term we have
\begin{prop} \label{prop:constant} For any  $A>0$, as $T \rightarrow \infty$, we have
$$
|\langle E_A^2(\cdot, \tfrac{1}{2} +iT), (\tfrac{3}{\pi})^{1/2}\rangle|^2
\sim  \frac{12}{\pi} (\log T)^2.
$$
\end{prop}

\begin{proof}
We start with the Maass-Selberg relation: 
for $s_1 \neq s_2$, $s_1+s_2 \neq 1$, we have that
$$
\int_{\mathcal{F}} E_A(z, s_1) E_A(z, s_2)  d \mu z=\frac{A^{s_1+s_2-1} -\varphi(s_1) \varphi(s_2) A^{1-s_1-s_2}}{s_1+s_2-1}  + \frac{\varphi(s_2)A^{s_1-s_2}  - \varphi(s_1)A^{s_2-s_1}}{s_1-s_2},
$$
where $\varphi(s)=\frac{\xi(2s-1)}{\xi(2s)}$. We apply it to $s_1=\frac{1}{2} + \delta +iT   \neq s_2=\frac{1}{2}  +iT$ and take the limit $\delta \rightarrow 0$, getting
$$
\varphi(\tfrac{1}{2} +iT) \left[
2  \log A - \frac{\varphi'}{\varphi}(\tfrac{1}{2} +iT)
\right]
+ \frac{A^{2iT}  -\varphi(\frac{1}{2} +iT)^2 A^{-2iT}}{2iT}.
$$
Using Stirling's approximation $\frac{\Gamma'}{\Gamma}(\frac{1}{2} +iT)=\log T +O(1)$ and Vinogradov's estimate $\frac{\zeta'}{\zeta}(\frac{1}{2} +iT) \ll (\log T)^{\frac{2}{3} + \epsilon}$, we obtain 
$$
- \frac{\varphi'}{\varphi}(\tfrac{1}{2} +iT)= 2 \log T + O((\log T)^{\frac{2}{3} + \epsilon}).
$$
Hence as $T \rightarrow \infty$,
$$
\langle E_A^2(\cdot, \tfrac{1}{2} +iT), 1 \rangle= \int_{\mathcal{F}} E_A^2(\cdot, \tfrac{1}{2} +iT) d \mu z  \sim  2 \, \varphi(\tfrac{1}{2} +iT) \log T.
$$
The result follows since $|\varphi(\tfrac{1}{2} +iT)|=1$.

\end{proof}

For $u_j$ even we have by Rankin-Selberg integration (also called unfolding), cf. \cite[Lemma 4.1]{DK1}, that
\begin{equation} \label{eq:Rankin-Selb}
\langle E^2(\cdot, \tfrac{1}{2} +iT), u_j\rangle
=\frac{{\rho_j(1)}}{2} \frac{\Lambda(\frac{1}{2}, u_j) \Lambda(\frac{1}{2} + 2i T, u_j)}{\xi^2(1 +2i T)},
\end{equation}
while for $u_j$ odd the triple product vanishes. The following asymptotic formula is the main result of \cite{DK2}:

\begin{prop} \cite[Theorem 1.2]{DK2} \label{prop:IMRN_main} As $T \rightarrow \infty$, we have
\begin{equation}  \label{eq:sum_E^2_discrete}
\sum_{j \ge 1} \left|  
 \langle E^2( \cdot, \tfrac{1}{2} + iT), u_j \rangle \right|^2
 =\sum_{j \ge 1} \frac{{|\rho_j(1)|^2}}{4} \frac{\Lambda^2(\frac{1}{2}, u_j) |\Lambda(\frac{1}{2} + 2i T, u_j)|^2}{|\xi(1 +2i T)|^4} \sim \frac{48}{\pi} \log^2 T.
\end{equation}
\end{prop}

Next we need the asymptotic evaluation as $T \rightarrow \infty$ of the integral
\begin{align*}
\langle H_A, H_A \rangle
& =\int_{\mathcal{C}_A} 4  |e(y, \tfrac{1}{2} + iT) E_A(z, \tfrac{1}{2} + iT)|^2 d\mu(z)  \\
&  =\int_{\mathcal{C}_A}   4 \left( 2 y  +  \bc y^{1-2iT} + \overline{\bc} y^{1+2iT} \right)  |E_A(z, \tfrac{1}{2} + iT)|^2   d\mu(z).
\end{align*}
Since $\xi(1+2iT)E_A(z, \tfrac{1}{2} + iT)$ is real-valued, this integral is \emph{exactly}  equal $\frac{2}{3}$ times the integral at the bottom of page 251 in \cite{DK1}. The asymptotic  formula for that integral was calculated there, in \cite[formula (5.4)]{DK1}. Therefore, we get the following proposition:

\begin{prop} \label{prop:H_A_H_A} For any $A>1$, as $T \rightarrow \infty$, we have
$$
\langle H_A, H_A \rangle \sim \frac{24}{\pi} (\log T)^2.
$$
\end{prop}


Finally,  the two continuous spectrum contributions  in Lemma \ref{lema:spectral_expansion} are asymptotically negligible. Namely,
Spinu in \cite[Theorem 3.3]{S} proves that there exists a positive number $\delta > 0$ such that, as $T \rightarrow \infty$,
\begin{equation} \label{bound_E_A^2_continuous}
\frac{1}{4 \pi} \int_{\mathbb{R}} |\langle E_A^2(\cdot, \tfrac{1}{2} +iT), E(\cdot, \tfrac{1}{2} +it)\rangle|^2 dt \le 108 A +O(T^{-\delta}).
\end{equation}
In particular, this bound is $O(1)$.
Moreover, Spinu proves in \cite[Proposition 3.6]{S} the upper bound
\begin{equation} \label{bound_H_A_continuous}
\int_{\R}  \left| \langle  H_A, E(\cdot, \tfrac{1}{2} + it)   \rangle 
  \right|^2  dt
=\int_{\R}  \left|  
\int_{\mathcal{C}_A} 
2  e(y, \tfrac{1}{2} + iT) E_A(z, \tfrac{1}{2} + iT) \overline{E(z, \tfrac{1}{2} + it)}
 d \mu z
  \right|^2  dt \ll T^{-\frac{1}{6}}.
\end{equation}
In conclusion, we have reduced the task in proving Proposition \ref{the-prop} to finding an asymptotic formula for the average value of the cross-term:
\[
\int_{-\infty}^{\infty} h(A) \ \Xi(A) \ dA
\]
as $T\to\infty$.

\section{The cross-term }

\subsection{Reduction of the cross-term to the bulk range}  

Starting with (\ref{eq:Rankin-Selb}) and analyzing the size of the corresponding gamma functions, one can see \cite[Section 5.1.1]{S}  that the sum on the left hand side of (\ref{eq:sum_E^2_discrete}) is supported on $|t| <2 T + T^{\epsilon}$ for any fixed $0<\epsilon<\frac{1}{100}$, up to an error of $O(T^{-100})$ say. Moreover, it has been shown that 
$$
\left( \sum_{|t_j| < T^{1 - \epsilon}}   + \sum_{2T - T^{1- \epsilon}<  |t_j| < 2T + T^{ \epsilon}} \right)
  |\langle E^2( \cdot, \tfrac{1}{2} + iT), u_j \rangle|^2  \ll  T^{-\delta}
$$
for some $\delta>0$. 
For this see \cite[Sections 3.6 and 3.7]{H}. For brevity, we denote the spectral sum on the left hand side by $\sum^{\flat}$. By Cauchy-Schwarz, we have that the corresponding part of the cross-term $\Xi(A)$ is also negligible:
$$
\sideset{}{^\flat}\sum
  \overline{\langle E^2( \cdot, \tfrac{1}{2} + iT), u_j \rangle }   \langle  H_A, u_j   \rangle
$$
$$
\ll  \left( \sideset{}{^\flat}\sum  |\langle E^2( \cdot, \tfrac{1}{2} + iT), u_j \rangle|^2 \right)^{1/2} \left( \sideset{}{^\flat}\sum   | \langle  H_A, u_j   \rangle|^2  \right)^{1/2} \ll T^{-\delta /2}  \langle   H_A, H_A   \rangle^{1/2}\ll T^{-\delta/3} ,
$$ 
since $\langle   H_A, H_A   \rangle^{1/2} \ll \log T$ by Proposition \ref{prop:H_A_H_A}. Therefore it suffices to restrict the cross-term to the range
\begin{align}  \label{eq:bulk}
T^{1 - \alpha}\le t_j \le  2T - T^{1- \alpha},
\end{align}
for any fixed $0< \alpha < \frac{1}{100}$. We refer to this interval of $t_j$ as the `bulk range', and we can essentially pick it out with the smooth function $W(t)$ constructed in \cite[Lemma 5.1]{BK}:
\begin{equation} \label{eq:Xi_with_W}
\Xi(A) \sim  
 \sum_{\substack{j \ge 1    }}  W(t_j) \overline{\langle E^2( \cdot, \tfrac{1}{2} + iT), u_j \rangle }   \langle  H_A, u_j   \rangle.
\end{equation}
Recall that the function $W$ is explicitly given by
\begin{align}
\label{w-function} W(t)=W_{\alpha}(t)=\left( 1 - \exp \left( - \left( \frac{t}{(2T)^{1-\alpha/2}} \right)^{2 \lceil 1000/ \alpha \rceil} \right) \right) 
\left( 1 - \exp \left( - \left( \frac{4T^2-t^2}{4T^{2-\alpha/2}} \right)^{2 \lceil 1000/ \alpha \rceil} \right) \right).
\end{align}
We have that $W(t)$ is $O(T^{-100})$ unless $T^{1 - \alpha}\le |t| \le  2T - T^{1- \alpha}$, while $W(t)=1+O(T^{-100})$ for $T^{1 - \frac{\alpha}{4}}\le |t| \le  2T - T^{1- \frac{\alpha}{4}}$. Note that $\alpha$ is the same parameter as in Proposition \ref{the-prop}.

\subsection{The projection onto a cusp form: a formula for $\langle H_A, u_j \rangle$}

Let $u_j(z)$ be an even or odd Hecke-Maass cusp form for the group $\Gamma$. 
Using Fourier expansions in the cuspidal region $\mathcal{C}_A$, we get for any $A>1$
\begin{align*} 
\langle H_A, u_j \rangle
& =
\int_{\mathcal{C}_A}  2 e(y, \tfrac{1}{2} + iT) E_A(z, \tfrac{1}{2} + iT)  u_j(z) d \mu z \\
&
= \int_A^{\infty} \int_{-1/2}^{1/2} 2  e(y, \tfrac{1}{2} + iT) \frac{2}{\xi(1+ 2 iT)} \sum_{n \neq 0} \tau_{}(|n|, T) \sqrt{y} K_{iT}(2 \pi |n| y ) e(nx)    \nonumber \\
& \quad \times \rho_j(1)  \sum_{m \neq 0} \lambda_j(m)  \sqrt{y}  K_{i t_j}(2 \pi |m| y) e(mx)  dx \frac{dy}{y^2}    
 \\
&
=\frac{4 \rho_j(1)}{\xi(1+ 2 iT)}  \sum_{n \neq 0} \tau_{}(|n|, T)  \lambda_j(-n)    \int_A^{\infty}   e(y, \tfrac{1}{2} + iT)    K_{iT}(2 \pi |n| y )       K_{i t_j}(2 \pi |n| y)  \frac{dy}{y}. 
\end{align*}

From here we see that $\langle H_A, u_j \rangle$ 
vanishes for \emph{odd} $u_j$  (since then $\lambda_j(-n)=-\lambda_j(n)$), while for $u_j$ \emph{even}, we have
\begin{align*}
\langle H_A, u_j \rangle= \frac{8 \rho_j(1)}{\xi(1+ 2 iT)}  \sum_{n \ge 1} \tau_{}(n, T)  \lambda_j(n)    \int_A^{\infty} (y^{\frac{1}{2} + i T} + \bc \cdot y^{\frac{1}{2} - i T} )      K_{iT}(2 \pi n y )       K_{i t_j}(2 \pi n y)  \frac{dy}{y}.
\end{align*}
The integrals can be expressed as follows, where we follow the type of argument given in \cite[equations (2.26)--(2.28)]{S} . Denote for $x> 0$,
$$
g(x):=\int_x^{\infty}  y^{\frac{1}{2} + i T} K_{iT}(y)K_{it_j}(y) \frac{dy}{y},
$$
and note that this converges absolutely since $K_{ir}(y)\ll_r e^{-y}$ for $r\ge 1$ and $y>0$ (see for example \cite[equations (14) and (25)]{Bo}). Thus $g(x)\ll_T  e^{-x}$ for $x>0$, and so the Mellin transform
\[
G(s)  =\int_0^{\infty}  g(x) x^{s-1} dx
\]
converges absolutely for $\Re(s)>0$. Using integration by parts, we have
\begin{align*}
G(s) & =\int_0^{\infty}  g(x) x^{s-1} dx=-\frac{1}{s} \int_0^{\infty} g'(x) x^s dx=\frac{1}{s} \int_0^{\infty} x^{s+ \frac{1}{2} + i T} K_{iT}(x)K_{it_j}(x)  \frac{dx}{x} \\
& =\frac{2^{s-\frac{5}{2}+iT}}{s } 
\frac{ \prod_{\pm}
\Gamma\left(\frac{s+ \frac{1}{2} + 2 i T \pm i t_j}{2}\right)  
\Gamma\left(\frac{s+ \frac{1}{2} \pm i t_j}{2}\right) }{\Gamma(s+ \frac{1}{2} + i T)},
\end{align*}
where we used
the Mellin-Barnes formula \cite[6.576.4]{GR}, 
$$
\int_0^{\infty} x^s K_{\mu}(x)  K_{\nu}(x) \frac{dx}{x}=\frac{2^{s-3}}{\Gamma(s)} \prod_{\pm, \pm} \Gamma\left(\frac{s \pm \mu \pm \nu}{2} \right).
$$
Then by the inverse Mellin transform we have for $\sigma>0$ that
$g(x)=\frac{1}{2 \pi i} \int_{(\sigma)} G(s) x^{-s} ds$, which converges absolutely since $G(s)$ decays exponentially as $|\Im(s)|\to \infty$, by Stirling's approximation. Hence we get that for $u_j$ even and $\sigma>0$, we have
\begin{align*}
&\int_{-\infty}^\infty h(A) \ \langle H_A, u_j \rangle \ dA\\
&=\frac{ \rho_j(1)}{\xi(1+ 2 iT)}    \sum_{n \ge 1}  \frac{\tau_{}(n, T)  \lambda_j(n)}{( \pi n)^{\frac{1}{2} + i T }} 
 \frac{1}{2 \pi i} \int_{(\sigma)} \int_{-\infty}^\infty \frac{h(A)}{(A \pi n)^{s}} dA
\frac{ \prod_{\pm}
\Gamma\left(\frac{s+ \frac{1}{2} + 2 i T \pm i t_j}{2}\right)  
\Gamma\left(\frac{s+ \frac{1}{2} \pm i t_j}{2}\right) }{\Gamma(s+ \frac{1}{2} + i T)}     \frac{ds}{s }   \nonumber
\\
& + \bc \cdot \frac{ \rho_j(1)}{\xi(1+ 2 iT)}    \sum_{n \ge 1}  \frac{\tau_{}(n, T)  \lambda_j(n)}{( \pi n)^{\frac{1}{2} - i T }} 
 \frac{1}{2 \pi i} \int_{(\sigma)} \int_{-\infty}^\infty \frac{h(A)}{(A \pi n)^{s}} dA
\frac{ \prod_{\pm}
\Gamma\left(\frac{s+ \frac{1}{2} - 2 i T \pm i t_j}{2}\right)  
\Gamma\left(\frac{s+ \frac{1}{2} \pm i t_j}{2}\right) }{\Gamma(s+ \frac{1}{2} - i T)}   \frac{ds}{s }.
\end{align*}
We introduce the following notation
\begin{align} 
\label{cHdef} &\cH_{+}(s, t)=  \tilde{h}(1-s) \frac{ \prod_{\pm}
\Gamma\left(\frac{s+ \frac{1}{2} + 2 i T \pm i t}{2}\right)  
\Gamma\left(\frac{s+ \frac{1}{2} \pm i t}{2}\right) }{\Gamma(s+ \frac{1}{2} + i T)  \Gamma(\frac{1}{2} +i T) \Gamma(\frac{1}{2} +i t)},
\\
\nonumber &\cH_{-}(s, t)=  \tilde{h}(1-s) \frac{ \prod_{\pm}
\Gamma\left(\frac{s+ \frac{1}{2} - 2 i T \pm i t}{2}\right)  
\Gamma\left(\frac{s+ \frac{1}{2} \pm i t}{2}\right) }{\Gamma(s+ \frac{1}{2} - i T) \Gamma(\frac{1}{2} +i T) \Gamma(\frac{1}{2} +i t) },
\end{align}
where
\begin{align}
\label{htilde} \tilde{h}(s)= \int_{-\infty}^\infty h(A) (\pi A)^{s-1} dA
\end{align}
for any $s\in\mathbb{C}$ (note that our definition of $\tilde{h}$ does not quite coincide with the Mellin transform).
Then we denote for $\sigma >0$ and $x\ge 1$,
\begin{align}
\label{cV-def} \cV_{\pm}(x, t):= \frac{1}{2 \pi i} \int_{(\sigma)} \cH_{\pm}(s, t) x^{-s} \frac{ds}{s},
\end{align}
and finally write
\begin{align}
 \int_{-\infty}^\infty h(A)  \langle H_A, u_j \rangle dA
&
=\frac{ \rho_j(1)   \Gamma(\frac{1}{2} +i t_j)}{\zeta(1+ 2 iT)}    
 \sum_{n \ge 1}  \frac{\tau_{}(n, T)  \lambda_j(n)}{ n^{\frac{1}{2} + i T }} 
\cV_{+}(n, t_j)
 \nonumber
\\
& + \bc \cdot \frac{ \pi^{2iT} \rho_j(1)  \Gamma(\frac{1}{2} +i t_j)}{\zeta(1+ 2 iT)}      \sum_{n \ge 1}   \frac{\tau_{}(n, T)  \lambda_j(n)}{  n^{\frac{1}{2} - i T }} 
\cV_{-}(n, t_j).     \label{eq:H_Dirichlet_series}
\end{align}

The following result, afforded by the averaging over $A$, will allow us to restrict attention to small values of $|s|$.
\begin{lem}\label{small-s}
For $\Re(s)$ fixed and $|s|\ge T^\alpha$, we have
\[
\tilde{h}(s) \ll T^{-100}.
\]
\end{lem}
\proof This follows by repeatedly integrating by parts with respect to $A$ in \eqref{htilde} and using \eqref{h-derivatives}.
\endproof

\begin{rem}
Using the generalized  Ramanujan identity
$$
\sum_{n \ge 1}  \frac{\tau_{}(n, T)  \lambda_j(n)}{ n^{\frac{1}{2} + s \pm i T }}
=\frac{L(\frac{1}{2} + s, u_j) L(\frac{1}{2} + s \pm 2iT, u_j)}{\zeta(1 + 2s \pm  2iT )}
$$
we could rewrite (\ref{eq:H_Dirichlet_series}) in terms of $L$-functions. However we will not use this identity, and instead work with (\ref{eq:H_Dirichlet_series}) directly. 
\end{rem}

\subsection{The cross-term via Dirichlet series}

We insert (\ref{eq:Rankin-Selb}) and  (\ref{eq:H_Dirichlet_series}) into (\ref{eq:Xi_with_W}) to get
\begin{align*}
&\int_{-\infty}^\infty h(A) \ \Xi(A) \ dA \sim  
 \sum_{\substack{j \ge 1   \\   u_j \; \text{even}}}  W(t_j) \frac{{ \overline{\rho_j(1)}}}{2} \frac{\Lambda(\frac{1}{2}, u_j) \Lambda(\frac{1}{2} - 2i T, u_j)}{\xi^2(1 -2i T)} \int_{-\infty}^\infty h(A)   \langle  H_A, u_j   \rangle dA\\ &=\sum_{\substack{j \ge 1   \\   u_j \; \text{even}}}  W(t_j) \frac{{  
|\rho_j(1)|^2}}{2} 
\frac{    |\Gamma(\frac{1}{2} +i t_j)|^2 }{\zeta(1+ 2 iT)}  
\frac{\prod_{\pm}  \Gamma\left(\frac{ \frac{1}{2} \pm i t_j}{2}\right) 
\Gamma\left(\frac{ \frac{1}{2} - 2 i T \pm i t_j}{2}\right)  
}{\Gamma^2(\frac{1}{2} -iT)   \Gamma(\frac{1}{2} -i t_j)} 
\frac{L(\frac{1}{2}, u_j) L(\frac{1}{2} - 2i T, u_j)}{\zeta^2(1 -2i T)}\\
&\times  
\left[ 
 \sum_{n \ge 1}  \frac{\tau_{}(n, T)  \lambda_j(n)}{ n^{\frac{1}{2} + i T }} 
\cV_{+}(n, t_j)
  + \bc  \pi^{2iT}       \sum_{n \ge 1}   \frac{\tau_{}(n, T)  \lambda_j(n)}{  n^{\frac{1}{2} - i T }} 
\cV_{-}(n, t_j)
\right].
\end{align*}
We introduce the notation
$$
\cH(t):=\frac{\prod_{\pm}  \Gamma\left(\frac{ \frac{1}{2} \pm i t}{2}\right) 
\Gamma\left(\frac{ \frac{1}{2} - 2 i T \pm i t}{2}\right)  
}{\Gamma^2(\frac{1}{2} -iT)   \Gamma(\frac{1}{2} -i t)},
$$
and recalling (\ref{rho_1}), we get
\begin{align}
\int_{-\infty}^\infty h(A) \ \Xi(A) \ dA \sim & \frac{\pi}{\zeta(1+ 2 iT)  \zeta^2(1 -2i T)}
\sum_{\substack{j \ge 1   \\   u_j \; \text{even}}}  W(t_j)  
\frac{  \cH(t_j)   }{L(1, \text{sym}^2 u_j)} 
   L(\tfrac{1}{2}, u_j) L(\tfrac{1}{2} - 2i T, u_j)  
   \nonumber  \\
& \times  
\left[ 
 \sum_{n \ge 1}  \frac{\tau_{}(n, T)  \lambda_j(n)}{ n^{\frac{1}{2} + i T }} 
\cV_{+}(n, t_j)
  + \bc  \pi^{2iT}       \sum_{n \ge 1}   \frac{\tau_{}(n, T)  \lambda_j(n)}{  n^{\frac{1}{2} - i T }} 
\cV_{-}(n, t_j)
\right].  \label{eq:cross_before_AFE}
\end{align}
We plan to apply the (same-sign) Kuznetsov trace formula,
so we \emph{extend} the sum over $u_j$ in (\ref{eq:cross_before_AFE}) to run over both \emph{even and odd} forms, which is legitimate since $L(\frac{1}{2}, u_j)=0$ for odd forms $u_j$
(so although the expression (\ref{eq:H_Dirichlet_series}) holds for even forms only, for odd forms it will be multiplied by 0).

\subsection{An approximate functional equation}

The $L$-function attached to  Hecke-Maass cusp form $u_j$ is defined 
for $\Re(s)>1$ by
$$
L(s, u_j) =\sum_{n \ge 1}   \frac{\lambda_j(n)}{n^s}
$$
and  for  \emph{even} $u_j$ it satisfies the functional equation 
$$
L(s, u_j) \Gamma_{\R}(s+ i t_j) \Gamma_{\R}(s-  i t_j)
=L(1-s, u_j) \Gamma_{\R}(1-s+ i t_j) \Gamma_{\R}(1-s-  i t_j),
$$
where $\Gamma_{\R}(s)=\pi^{-\frac{s}{2}}  \Gamma(\tfrac{s}{2})$.
By a standard procedure (e.g. \cite[Section 5.2]{IK}) and using the Hecke multiplicative relation $\lambda_j(n)\lambda_j(m)=\sum_{k| (n, m)}\lambda_j(\frac{nm}{k^2})$, we obtain the following approximate functional equation for \emph{even} Hecke-Maass cusp forms $u_j$:
\begin{align} \label{AFE_0_-2iT}
L(\tfrac{1}{2}, u_j) L(\tfrac{1}{2} -2iT, u_j)
= & \sum_{n\ge 1} \sum_{k \ge 1} \frac{\lambda_j(n) \tau(n, -T )}{n^{\frac{1}{2}-iT}  k^{1  -2iT}}  V_{+}(k^2 n, t_j)   \nonumber   \\
&  + 
\sum_{n\ge 1} \sum_{k \ge 1} \frac{\lambda_j(n) \tau(n, T )}{n^{\frac{1}{2}+iT }  k^{1 +2iT}}  V_{-}(k^2 n, t_j),
\end{align}
where the weight functions are defined for $\sigma >0$ and $x\ge 1$ by
\begin{align}
\label{V+def} V_{\pm}(x, t_j)=\frac{1}{2 \pi i} \int_{(\sigma)} e^{w^2} x^{-w}  G_{\pm,\frac12}(w,t_j) \frac{dw}{w}
\end{align}
and where 
\begin{align*}
&G_{+,a}(w,t_j) = \frac{\prod_{\pm} \Gamma_{\R}(a  +w \pm i t_j) \Gamma_{\R}(a-2iT +w \pm i t_j)}{\prod_{\pm} \Gamma_{\R}(a  \pm i t_j) \Gamma_{\R}(a-2iT  \pm i t_j)},\\
&G_{-,a}(w,t_j) = \frac{\prod_{\pm} \Gamma_{\R}(a    +w \pm i t_j) \Gamma_{\R}(a+2iT +w \pm i t_j)}{\prod_{\pm} \Gamma_{\R}(a  \pm i t_j) \Gamma_{\R}(a-2iT  \pm i t_j)}
\end{align*}
for $a>0$.
For $u_j$ odd we
have the functional equation 
$$
L(s, u_j) \Gamma_{\R}(1+ s+ i t_j) \Gamma_{\R}(1+s-  i t_j)
= - L(1-s, u_j) \Gamma_{\R}(2-s+ i t_j) \Gamma_{\R}(2-s-  i t_j)
$$
and consequently the approximate functional equation
$$
L(\tfrac{1}{2}, u_j) L(\tfrac{1}{2} -2iT, u_j)
= \sum_{\pm} \sum_{n\ge 1} \sum_{k \ge 1} \frac{\lambda_j(n) \tau(n, T )}{n^{\frac{1}{2} \mp iT}  k^{1  \mp 2iT}}  V_{\pm}^{\text{odd}}(k^2 n, t_j)    
$$
has slightly different weight functions
$$
V_{\pm}^{\text{odd}}(x, t_j)=\frac{1}{2 \pi i} \int_{(\sigma)} e^{w^2} x^{-w} G_{\pm,\frac32}(w,t_j)  \frac{dw}{w}.
$$

Next we show, following \cite[Section 2.2]{DK2}, that for our purposes, we can use the approximate functional equation \eqref{AFE_0_-2iT} for odd forms as well, because although the weight functions differ, this difference contributes a negligible amount overall to \eqref{eq:cross_before_AFE}. Namely, it suffices to bound by a negative power of $T$ the sum
\begin{align*}
\sum_{\substack{j \ge 1  }} W(t_j) |\cH(t_j)|  
 \left| \sum_{\pm } \sum_{n,k\ge 1} \frac{\lambda_j(n) \tau(n, T )}{n^{\frac{1}{2} \mp iT}  k^{1   \mp 2iT}} 
 V_{\pm}^{\rm diff}(k^2 n, t_j)   \right|
\left|
\sum_{m \ge 1}  \frac{  \lambda_j(m) \tau(m, T) }{ m^{\frac{1}{2} \pm i T }} \cV_{\pm}(m, t_j)   
 \right|,
\end{align*}
where 
\[
V_{\pm}^{\rm diff}(k^2 n, t_j)=V_{\pm}^{\rm odd}(k^2 n, t_j)  - V_{\pm}(k^2 n, t_j). 
\]
We will see by Lemma \ref{lem_support_V} and Lemma \ref{lem_support_cV} below that we can restrict the sums above to $nk^2\le T^{2+\epsilon}$ and $m\le T^{1+\epsilon}$, up to an error of $O(T^{-100})$. We can restrict the $w$-integrals in the weight functions $V_{\pm}$ and $V_{\pm}^{\rm odd}$ to $|w| < T^{\epsilon}$ by the rapid decay of $e^{w^2}$ in vertical lines, and the $s$-integral in the weight function $\cV_{\pm}$ to $|s|<T^\alpha$ by Lemma \ref{small-s}, up to an error of $O(T^{-100})$. Thus it suffices to bound
\begin{align}
\label{discrep} &\sum_{\substack{j \ge 1  }} W(t_j)  |\cH(t_j)| |\cH(s,t_j)|  \left| G_{\pm,\frac12}(w,t_j)-G_{\pm,\frac32}(w,t_j) \right|  \\
\nonumber  \times&
 \left| \sum_{\pm } \sum_{n,k\le T^{2+\epsilon}} \frac{\lambda_j(n) \tau(n, T )}{n^{\frac{1}{2} \mp iT+w}  k^{1   \mp 2iT+2w}} 
\right|
\left|
\sum_{m \le T^{1+\epsilon}}  \frac{  \lambda_j(m) \tau(m, T) }{ m^{\frac{1}{2} \pm i T + s }}   
 \right|
\end{align}
for $\Re(s),\Re(w)=\epsilon$ with $|w|<T^\epsilon$ and $|s|<T^{\alpha}$. In this range we have by Stirling's approximation,
\[
G_{\pm,\frac12}(w)-G_{\pm,\frac32}(w)\ll |t|^{-1}\ll T^{-1+\alpha}.
\]
We are summing over $O(T^2)$ forms, and will see in \eqref{Hbound} that $\cH(t_j) \ll T^{-1+\frac{\alpha}{2}}$ and $\cH(s,t_j)   \ll T^{-1+\frac{\alpha}{2}+\epsilon}$. Thus by the spectral large sieve \cite[Theorem 7.24]{IK} we get that \eqref{discrep} is $O(T^{-1+2\alpha+\epsilon})$, which is a negative power of $T$.

\subsection{Decomposition of $\int_{-\infty}^\infty h(A) \ \Xi(A) \ dA$ }

Inserting the approximate functional  equation developed in the previous subsection into \eqref{eq:cross_before_AFE}, we arrive at
\begin{align}
 \int_{-\infty}^\infty h(A)  \Xi(A)   dA \sim
& 
\frac{\pi}{\zeta^2(1 -2i T) \zeta(1 +2i T)}
\sum_{\substack{j \ge 1   }}  \frac{ W(t_j) \cH(t_j)}{L(1, \text{sym}^2 u_j)} 
  \sum_{\pm } \sum_{n\ge 1} \sum_{k \ge 1} \frac{\lambda_j(n) \tau(n, T )}{n^{\frac{1}{2} \mp iT}  k^{1   \mp 2iT}}  V_{\pm}(k^2 n, t_j)  
  \nonumber \\ 
  &
\times
\left( 
\sum_{m \ge 1}  \frac{  \lambda_j(m) \tau(m, T) }{ m^{\frac{1}{2} + i T }} \cV_{+}(m, t_j)   
 + \bc \pi^{2iT}    
 \sum_{m \ge 1}   \frac{  \lambda_j(m) \tau(m, T)}{ m^{\frac{1}{2} - i T }}  
 \cV_{-}(m, t_j) 
 \right) \nonumber \\
 &
 =: \frac{\pi}{\zeta^2(1 -2i T) \zeta(1 +2i T)} \left( \Xi_{++} + \Xi_{-+} + \bc \pi^{2iT}  \Xi_{+-} + \bc \pi^{2iT} \Xi_{--}  \right), \label{eq_final_expr_before_Kuznetsov}
\end{align}
where we define
\begin{equation} \label{eq:Xi++}
\Xi_{++}=
   \sum_{n\ge 1} \sum_{k \ge 1} \frac{ \tau(n, T )}{n^{\frac{1}{2} - iT}  k^{1   - 2iT}}    
   \sum_{m \ge 1}  \frac{   \tau(m, T) }{ m^{\frac{1}{2} + i T }}
 \sum_{\substack{j \ge 1   }} \frac{\lambda_j(n) \lambda_j(m) }{L(1, \text{sym}^2 u_j)}   
  W(t_j) \cH(t_j)  V_{+ }(k^2 n, t_j) \cV_{+}(m, t_j),
\end{equation}
and analogously for the other three sums.

Since the size of the desired main term in Proposition \ref{the-prop} is $\gg \hat{h}(0) \log^2 T$, from \eqref{eq_final_expr_before_Kuznetsov} we should keep in mind that it suffices to obtain asymptotic expressions for each sum $ \Xi_{\pm \pm}$ up to an error of $o(\hat{h}(0) |\zeta(1+2iT)|^3 \log^2 T)$. For example, the error term $O(T^{-\alpha})$ would suffice since  $| \zeta(1 +2i T)|\ll \log T$ by classical estimates and $\hat{h}(0) \asymp T^{-\frac{\alpha}{2}}$ by assumption.

\subsection{Analysis of the weight functions} \label{Subsection_weights_support}

We make a note of the leading terms in the Stirling expansions of the various gamma function ratios implicit in \eqref{eq_final_expr_before_Kuznetsov}. Throughout we will suppose that $t$ lies the bulk range $T^{1-\alpha} < |t| < 2T- T^{1-\alpha}$. By Stirling's approximation applied to $G_{\pm, \frac12}(w,t)$ for $\Re(w)>0$ fixed and $|w|<T^{\epsilon}$, we have
\begin{align}
\label{V+}  V_{\pm}(x, t)  =V_{\pm}(t) \frac{1}{2\pi i} \int_{(\sigma)} e^{w^2 \mp i \frac{\pi}{2}w}  
\left( \frac{|t|(4T^2-t^2)^{\frac{1}{2}}}{4 \pi^2  x} \right)^{w} \frac{dw}{w} +\ldots
\end{align}
for $\sigma>0$ and $x\ge 1$, where
\begin{align}
\nonumber &V_+(t) =1,\\
&  V_{-}(t):=
(2 \pi e)^{-4iT} e^{-i \frac{\pi}{2}} |2T+t|^{i(2T+t)} |2T-t|^{i(2T-t)}.
 \label{V-size_1}
\end{align}
are functions of modulus $1$. We immediately get 
 \begin{lem}\label{lem_support_V} For values of $t$ in the bulk range and $x\ge 1$, we have
$$
V_{\pm}(x, t) \ll \left( \frac{|t|(4T^2-t^2)^{\frac{1}{2}}}{x} \right)^{\sigma} \ll \left( \frac{T^2}{x} \right)^{\sigma}
$$
for any fixed $\sigma >0$. Further for any fixed $\epsilon>0$, we have $V_{\pm}(x, t) \ll T^{-100}$ for $x \ge T^{2+ \epsilon}$.
\end{lem}

\

By Stirling's approximation, for $\Re(s)=\sigma>-\frac12$ fixed and $|s|<T^{\alpha}$, we have the bounds
\begin{align}
\label{Hbound} &|\cH(t)|^2 \ll |t|^{-1} (4T^2-t^2)^{-\frac12} \\
\label{Hsbound} &|\cH_{\pm}(s, t)|^2 \ll |t|^{-1} (4T^2-t^2)^{-\frac12}  \left( \frac{|t|(4T^2-t^2)^{\frac{1}{2}}}{4 T} \right)^{\sigma}.
\end{align}
Using this, we obtain the size and support of the weight functions $\cV_{\pm}(x, t)$.
\begin{lem} \label{lem_support_cV}
Let $0<\alpha<\frac{1}{100}$ be the parameter from equation \eqref{w-function}. For $t$ in the bulk range and any fixed $\epsilon>0$ and $\sigma>0$, we have 
\begin{equation} \label{eq:upper_cal_V}
\cV_{\pm}(x, t) \ll T^{-1 +\frac{\alpha}{2} + \epsilon} \left( \frac{T}{x} \right)^{\sigma}
\end{equation}
for $x\ge 1$.  For $x\ge T^{1+\epsilon}$, we have $\cV_{\pm}(x, t)\ll T^{-100}$.
\end{lem}
\proof
Recall that 
$$
\cV_{\pm}(x, t)= \frac{1}{2 \pi i} \int_{(\sigma)} \cH_{\pm}(s, t) x^{-s} \frac{ds}{s}
$$
for $\sigma>0$. Also recall by equation \eqref{cHdef} that $\cH_{\pm}(s, t)$ has a factor of $\tilde{h}(1-s)$, and by Lemma \ref{small-s}, this is $O(T^{-100})$ unless $|s|<T^\alpha$. So we restrict the integral above to $|s|<T^\alpha$. Then using \eqref{Hsbound}, we obtain the bound \eqref{eq:upper_cal_V}. If $x\ge T^{1+\epsilon}$, then we can take $\sigma$ large enough in \eqref{eq:upper_cal_V} to get $\cV_{\pm}(x, t) \ll T^{-100}$.
\endproof
The key point here is that while $V_{\pm}(n, t_j)$ is essentially supported for $n\le T^{2+\epsilon}$, the weight functions $\cV_{\pm}(n, t_j) $ are supported only for $n\le T^{1+\epsilon}$.

\

Finally, we will need the following leading terms from Stirling's approximation, for $\Re(s)>-\frac12$ fixed and $|s|<T^\alpha$.
\begin{align}
\label{H} &\cH(t)  \cH_{+}(s, t) = \frac{8\pi  \tilde{h}(1-s)}{|t|(4T^2-t^2)^{\frac{1}{2}}} \left( \frac{|t|(4T^2-t^2)^{\frac{1}{2}}}{4 T} \right)^{s} +\ldots, \\
\label{H-} &\cH(t)  \cH_{-}(s, t) V_{-}(t) =  \Big(\frac{T}{\pi^2 e}\Big)^{2iT} e^{-i\pi s}  \frac{8\pi  \tilde{h}(1-s)}{|t|(4T^2-t^2)^{\frac{1}{2}}} \left( \frac{|t|(4T^2-t^2)^{\frac{1}{2}}}{4 T} \right)^{s} +\ldots
\end{align}

\section{Kuznetsov’s trace formula}

We restate Kuznetsov's trace formula from \cite[Lemma 3.2]{BK}. Let $\phi(z)$ be an even, holomorphic function on $|\Im(z)|< \frac{1}{4} + \theta$ satisfying $|\phi(z)| \ll (1+ |z|)^{-2- \theta}$ on that strip, for some $\theta >0$. Then for all integers $n, m >0$, we have
\begin{align} \label{eq:Kuznetsov}
& \sum_{j \ge 1} \frac{\lambda_j(n) \lambda_j(m)}{L(1, \text{sym}^2 u_j)} \phi(t_j) + \int_{- \infty}^{\infty} \frac{\tau(n, t) \tau(m, -t)}{|\zeta(1+ 2it)|^2} \phi(t) \frac{d t}{2 \pi}  \nonumber \\
& 
= \delta_{n, m} \int_{- \infty}^{\infty}  \phi(t) \frac{d^*t}{2  \pi^2} + \sum_{c \ge 1} \frac{S(n, m; c)}{c} \int_{- \infty}^{\infty} \mathcal{J}\left( \frac{\sqrt{nm}}{c}, t \right) \phi(t) \frac{d^*t}{2 \pi},
\end{align}
where $d^*t=\tanh(\pi t) \, t  \, dt$ and
$\mathcal{J}(x, t)= \frac{2 i }{\sinh(\pi t)} J_{2it}(4 \pi x)$.

Note that the function
$\phi(t)=W(t) \cH(t)  V_{ \pm }(k^2 n, t) \cV_{\pm}(m, t)$ appearing in (\ref{eq_final_expr_before_Kuznetsov}), satisfies the above conditions.
We  apply (\ref{eq:Kuznetsov}) to each of the sums $\Xi_{\pm \pm}$. In particular, $\Xi_{++}$ transforms into the diagonal, the Eisenstein and the off-diagonal contribution:
$$
\Xi_{++}= \cD_{++} +  \cE_{++} + \cO_{++} .
$$
We use similar notation for the other $\Xi_{\pm \pm}$ sums.

\section{The diagonal contribution}

\subsection{Diagonal $\cD_{++}$}    \label{Section_diag_I}

Applying (\ref{eq:Kuznetsov}) to the inner sums in (\ref{eq:Xi++}) 
we get the diagonal contribution:
$$
\cD_{++}=
   \sum_{n\ge 1} \sum_{k \ge 1} \frac{ \tau(n, T )^2 }{n    k^{1   - 2iT}}      
    \int_{- \infty}^{\infty}   W(t) \cH(t) V_{+ }(n k^2, t) \cV_{+}(n, t) \frac{d^*t}{2  \pi^2}.
$$

Let $\sigma_1, \sigma_2  > 0$.
Using only the leading terms given in equations (\ref{V+}) and \eqref{H}, we have 
\begin{align*}
\cD_{++} & \sim  
\sum_{n\ge 1} \sum_{k \ge 1} \frac{ \tau(n, T )^2 }{n    k^{1   - 2iT}}      
    \int_{- \infty}^{\infty}   \frac{W(t)}{|t|(4T^2-t^2)^{\frac{1}{2}}}    \frac{1}{2 \pi i} \int_{(\sigma_1)} e^{w^2 -i \frac{\pi}{2}w}  
\left( \frac{|t|(4T^2-t^2)^{\frac{1}{2}}}{4 \pi^2  n k^2 } \right)^{w}
 \frac{dw}{w}   \\
 & \qquad
 \times 
 \frac{1}{2 \pi i}  \int_{(\sigma_2)}  8\pi  \tilde{h}(1-s) \left( \frac{|t|(4T^2-t^2)^{\frac{1}{2}}}{4 T} \right)^{s}  n^{-s} \frac{ds}{s}  \frac{d^*t}{2  \pi^2}.
 \end{align*}
 It suffices to only consider the contribution of the leading terms because by \eqref{stirling2}, the lower order terms are of similar shape, but much smaller. Rearranging, we have
 \begin{align*}
\cD_{++}  &
 \sim 
    \int_{- \infty}^{\infty}     \frac{W(t)}{|t|(4T^2-t^2)^{\frac{1}{2}}} 
    \frac{1}{2 \pi i} \int_{(\sigma_1)} e^{w^2 -i \frac{\pi}{2}w}  
\left( \frac{|t|(4T^2-t^2)^{\frac{1}{2}}}{4 \pi^2    } \right)^{w}
\sum_{k \ge 1} \frac{1}{ k^{1   - 2iT +2 w}}
    \\
 & \qquad
 \times 
  \frac{1}{2 \pi i} \int_{(\sigma_2)} 
8\pi  \tilde{h}(1-s) \left( \frac{|t|(4T^2-t^2)^{\frac{1}{2}}}{4 T} \right)^{s}
  \sum_{n\ge 1}  \frac{ \tau(n, T )^2 }{n^{1+w+s}   }       
 \frac{ds}{s} \frac{dw}{w}   \frac{d^*t}{2  \pi^2}.
\end{align*}
Writing $\tau(n, T )=\frac{\sigma_{2i T}(n)}{n^{iT}}$ and using Ramanujan's identity, where $a, b \in \mathbb{C}$,
\begin{equation} \label{eq:Ramanujan}
\sum_{n \ge 1} \frac{\sigma_a(n)  \sigma_b(n)}{n^s} =\frac{\zeta(s)\zeta(s-a)\zeta(s-b)\zeta(s-a-b)}{\zeta(2s -a-b)},
\end{equation}
we have
\begin{align*}
\cD_{++} &\sim 
    \int_{- \infty}^{\infty}     \frac{W(t)}{|t|(4T^2-t^2)^{\frac{1}{2}}} 
  \frac{1}{2 \pi i} \int_{(\sigma_1)} e^{w^2 -i \frac{\pi}{2}w}  \frac{1}{2 \pi i} \int_{(\sigma_2)} 
\left( \frac{|t|(4T^2-t^2)^{\frac{1}{2}}}{4    } \right)^{w+s}  \pi^{-2w}T^{-s}
\\
 & \times  8\pi  \tilde{h}(1-s) 
  \frac{\zeta(1   - 2iT +2 w) \zeta(1+w+s  )^2  \zeta(1+w+s +2iT) \zeta(1+w+s -2iT)}{\zeta(2+2w+2s )} 
  \frac{ds}{s} 
  \frac{dw}{w}     
    \frac{d^*t}{2  \pi^2}.
\end{align*}


Let us fix the values
$\sigma_1= \frac{1}{101}$ and $\sigma_2 = \frac{1}{100} > \sigma_1$, say. In the double complex integral
$$
  \frac{1}{2 \pi i} \int_{(\sigma_1)} e^{w^2 -i \frac{\pi}{2}w}  \frac{1}{2 \pi i} \int_{(\sigma_2)} 
\left( \frac{|t|(4T^2-t^2)^{\frac{1}{2}}}{4     } \right)^{w+s}  \pi^{-2w}T^{-s}
8\pi  \tilde{h}(1-s)
$$
$$
 \times  
  \frac{\zeta(1   - 2iT +2 w) \zeta(1+w+s  )^2  \zeta(1+w+s +2iT) \zeta(1+w+s -2iT)}{\zeta(2+2w+2s )} 
  \frac{ds}{s} 
  \frac{dw}{w},
$$
we first move the $w$-contour to the left, to the line $\Re(w)=-\sigma_1$. We cross simple poles at $w=0$ and $w= iT$, getting
$$
\mathcal{R}_{w=0} + \mathcal{R}_{w=iT} + \mathcal{I}(t),
$$ 
where $\mathcal{I}(t)$ denotes the shifted double integral. The residue at $w=0$ is
\begin{equation} \label{eq:Res_w=0_I}
\mathcal{R}_{w=0}=
\zeta(1   - 2iT )
\frac{1}{2 \pi i} \int\limits_{(\sigma_2)}     
  8\pi  \tilde{h}(1-s) \left( \frac{|t|(4T^2-t^2)^{\frac{1}{2}}}{4 T} \right)^{s}
  \frac{ \zeta(1+s  )^2  \zeta(1+s +2iT) \zeta(1+s -2iT)}{\zeta(2+2s )} 
  \frac{ds}{s}, 
\end{equation}
the residue $\mathcal{R}_{w=iT}$ is negligible because of the $e^{w^2}$ factor, and the shifted integral is
\begin{align}
\nonumber \mathcal{I}(t) &=  \frac{1}{2 \pi i} \int_{(-\sigma_1)} e^{w^2 -i \frac{\pi}{2}w}  \frac{1}{2 \pi i} \int_{(\sigma_2)} 
\left( \frac{|t|(4T^2-t^2)^{\frac{1}{2}}}{4     } \right)^{w+s}  \pi^{-2w}T^{-s} \zeta(1   - 2iT +2 w)
\\
\label{eq_line_with_zeta_factors_I} & \times  
 8\pi  \tilde{h}(1-s) \frac{ \zeta(1+w+s  )^2  \zeta(1+w+s +2iT) \zeta(1+w+s -2iT)}{\zeta(2+2w+2s )} 
  \frac{ds}{s} 
  \frac{dw}{w}.
\end{align}
First, we will show that
$$
\int_{- \infty}^{\infty}  \frac{W(t)}{|t|(4T^2-t^2)^{\frac{1}{2}}}
   \mathcal{I}(t) \frac{d^*t}{2  \pi^2} \ll T^{-\alpha}, 
$$
which is an admissible error term by the remark following equation \eqref{eq:Xi++}. Because of the decay of the $e^{w^2}$ factor, we can restrict $w$-integration  to $|w| < T^{\epsilon}$, up to a negligible error. Weyl's subconvexity bound $\zeta(\frac{1}{2}+it) \ll |t|^{\frac{1}{6}+\epsilon}$ and the Phragm\'{e}n-Lindel\"{o}f principle together give that $\zeta(1- \sigma +it) \ll |t|^{\frac{\sigma}{3}+ \epsilon}$, for $0 \le \sigma \le \frac{1}{2}$. Using this and recalling that $t$ is restricted to the bulk range $T^{1-\alpha} < |t| < 2T -T^{1-\alpha}$,  on the new line $\Re(w)= - \sigma_1$ we get that
$$
\left( \frac{|t|(4T^2-t^2)^{\frac{1}{2}}}{4 \pi^2    } \right)^{w} \ll T^{-(2- \alpha) \sigma_1},   \qquad
\zeta(1   - 2iT +2 w) \ll T^{\frac{2 \sigma_1}{3} + \epsilon},
$$
for any fixed $\epsilon>0$. Moreover, the zeta-factors in the second line (\ref{eq_line_with_zeta_factors_I}) are absoletly convergent since
$\Re(1+w+s)=1 +\sigma_2 -\sigma_1=1+ \frac{1}{10100}$.
Hence
\begin{align*}
\int_{- \infty}^{\infty} &  \frac{W(t)}{|t|(4T^2-t^2)^{\frac{1}{2}}}  \mathcal{I}(t) \frac{d^*t}{2  \pi^2}\\
& \ll   T^{- \frac{4 \sigma_1}{3} + \alpha \sigma_1 + \epsilon}
\int_{- \infty}^{\infty}   \frac{W(t)}{|t|(4T^2-t^2)^{\frac{1}{2}}} 
\int_{- \infty}^{\infty}
    |  \tilde{h}(1-\sigma_2-iy) | \left( \frac{|t|(4T^2-t^2)^{\frac{1}{2}}}{4 T} \right)^{\sigma_2}
  \frac{dy}{1+|y|} d^*t\\
& \ll  T^{\sigma_2-1- \frac{4 \sigma_1}{3} + \alpha \sigma_1 + \epsilon}\int_{- \infty}^{\infty} W(t) dt  \int_{- \infty}^{\infty}   \frac{  |  \tilde{h}(1-\sigma_2-iy) |}{1+|y|} dy \\
&  \ll  T^{\sigma_2- \frac{4 \sigma_1}{3} + \alpha \sigma_1 + \epsilon} \ll T^{- \frac{97-300\alpha}{30300} +  \epsilon},
\end{align*}
which is $O(T^{-\alpha})$ for $\alpha$ small enough.

Next, in the residue $\mathcal{R}_{w=0}$, given in equation (\ref{eq:Res_w=0_I}),
we move the contour to the left to $\Re(s)=-\sigma_2 <0$, crossing two simple poles at $s=2iT$, $s=-2iT$ and a triple pole at $s=0$, getting:
$$
\mathcal{R}_{w=0}=\mathcal{R R}_{s=0} + \mathcal{R R}_{s=2iT} + \mathcal{R R}_{s=-2iT} + \mathcal{I R}(t),
$$
where the above notation refers to the sum of three residues and $\mathcal{I R}(t)$ denotes the integral (\ref{eq:Res_w=0_I}), but on the new line $\Re(s)=-\sigma_2 =-\frac{1}{100}$. The integral $\mathcal{I R}(t)$ is
$$
\ll |\zeta(1   - 2iT )|  \left| \frac{|t|(4T^2-t^2)^{\frac{1}{2}}}{4 T} \right|^{-\sigma_2}
 \int\limits_{(-\sigma_2)}     
    | \tilde{h}(1-s)| 
  \left| \zeta(1+s  )^2  \zeta(1+s +2iT) \zeta(1+s -2iT)
  \right| 
  \frac{|ds|}{|s|}. 
$$
By Lemma \ref{small-s}, we can restrict the integral to the interval $\Im(s) \in [-T^\alpha, T^\alpha]$, up to a negligible error. Then, using again Weyl's subconvexity 
$\zeta(1+s \pm 2iT) \ll T^{\frac{\sigma_2}{3} + \epsilon}$
for the last two zeta-factors and the classical bound $\zeta(1   - 2iT ) \ll \log T$, we get
$$
|\mathcal{I R}(t)| \ll   T^{\frac{2 \sigma_2}{3} -\sigma_2(1-\alpha)+ \epsilon} 
 \int_{-T^\alpha}^{T^\alpha}     
  \left| \zeta(1-\sigma_2 + iy  )  
  \right|^2  \frac{dy}{1+|y|} \ll T^{\frac{2 \sigma_2}{3} -\sigma_2(1-\alpha)+ \epsilon} \ll T^{-\frac{1}{100}(\frac{1}{3} -\alpha)+ \epsilon},
$$
which is $O(T^{-\alpha})$ for $\alpha$ small enough. 
Next we turn to the residues $\mathcal{R R}_{s=\pm 2iT}$ from the poles at $s=\pm 2iT$, but we can immediately say that these are negligible by Lemma \ref{small-s}. 

Finally, we are left with the contribution of the residue $\mathcal{R R}_{s=0}$ at the triple pole at $s=0$. This residue equals $\frac12$ times the second derivative of
\[
  8\pi  \tilde{h}(1-s) \left( \frac{|t|(4T^2-t^2)^{\frac{1}{2}}}{4 T} \right)^{s}
  \frac{ \zeta(1   - 2iT )  \zeta(1+s +2iT) \zeta(1+s -2iT)}{\zeta(2+2s )} 
\]
at $s=0$, plus other terms which are asymptotically smaller.    
 We recall the following classical estimates for the Riemann zeta-function on the edge of the critical strip (see \cite[Theorem 8.27, Theorem 8.29]{IK} and \cite[Lemma 4.3]{DK1}):
\begin{equation}   \label{eq:Vinogradov_bounds}
\zeta(1 \pm iT) \ll (\log T)^{\frac{2}{3}}, \qquad
 \frac{\zeta'}{\zeta}(1 \pm iT) \ll (\log T)^{\frac{2}{3}+\epsilon}, 
 \qquad   
 \frac{\zeta''}{\zeta}(1 \pm iT) \ll (\log T)^{\frac{4}{3}+\epsilon}.
\end{equation}
It follows that 
\[
\mathcal{R R}_{s=0}  \sim  4\pi \tilde{h}(1) \frac{\zeta(1 +2iT) \zeta(1 -2iT)^2}{\zeta(2)}  \log^2 \left(  \frac{|t|(4T^2-t^2)^{\frac{1}{2}}}{ T }\right).
\]
Thus, writing $\tilde{h}(1)=\hat{h}(0)$, 
\begin{align}
\cD_{++}
\nonumber & \sim 
\int_{- \infty}^{\infty}    \frac{W(t)}{|t|(4T^2-t^2)^{\frac{1}{2}}}  4\pi \hat{h}(0) \frac{\zeta(1 +2iT) \zeta(1 -2iT)^2}{\zeta(2)}  \log^2 \left(  \frac{|t|(4T^2-t^2)^{\frac{1}{2}}}{ T }\right)
  \frac{d^*t}{2  \pi^2}\\
\nonumber  & \sim \frac{24}{\pi^3} \hat{h}(0) \zeta(1 +2iT) \zeta(1 -2iT)^2 \int_{0}^{\infty}  \frac{W(t)}{(4T^2-t^2)^{\frac{1}{2}}}  \log^2 \left(  \frac{t(4T^2-t^2)^{\frac{1}{2}}}{ T }\right) dt\\
 \nonumber  & \sim \frac{24}{\pi^3} \hat{h}(0) \zeta(1 +2iT) \zeta(1 -2iT)^2 \int_{T^{1-\frac{\alpha}{4}}}^{2T-T^{1-\frac{\alpha}{4}}}  \frac{ \log^2(t)}{(4T^2-t^2)^{\frac{1}{2}}}  dt\\
 \nonumber&  \sim \frac{24}{\pi^3} \hat{h}(0) \zeta(1 +2iT) \zeta(1 -2iT)^2 \left(   \int_{T^{1-\frac{\alpha}{4}}}^{T^{1-\frac{1}{\sqrt{\log T}}}}  +  \int_{T^{1-\frac{1}{\sqrt{\log T}}}} ^{2T-T^{1-\frac{\alpha}{4}}} \right)  \frac{ \log^2(t)}{(4T^2-t^2)^{\frac{1}{2}}}  dt.
\end{align}
The contribution of the first integral in the last line is 
\[
\ll   T^{-\frac{1}{\sqrt{\log T}}} \hat{h}(0) |\zeta(1 +2iT)|^3 \log^2 T \ll  \hat{h}(0).
\]
In the second integral, we write $\log t = \log T  + O(\sqrt{\log T})$. The integral can then be asymptotically evaluated using the arcsine function, as in  \cite[Section 6]{BK}, to give
\begin{equation}  \label{eq:D_++}
\cD_{++}  
\sim \hat{h}(0)
\frac{ 12  }{ \pi^2}  \zeta(1 +2iT) \zeta(1 -2iT)^2 \log^2 T.
\end{equation}

\subsection{Diagonal $\cD_{-+}$}   \label{Section_diag_II}

The treatment of $\cD_{-+}$ has a crucial difference to that of $\cD_{++}$. Whereas in the previous subsection we needed Weyl strength subconvexity bounds for the Riemann zeta function, here we will need the much deeper sub-Weyl subconvexity.

The diagonal contribution arising by applying (\ref{eq:Kuznetsov}) to
$$
\Xi_{-+}
= 
  \sum_{n\ge 1} \sum_{k \ge 1} \frac{ \tau(n, T )}{n^{\frac{1}{2} + iT}  k^{1   + 2iT}}    
\sum_{m \ge 1}  \frac{   \tau(m, T) }{ m^{\frac{1}{2} + i T }}
\sum_{\substack{j \ge 1   }} \frac{ \lambda_j(n) \lambda_j(m)}{L(1, \text{sym}^2 u_j)}
W(t_j) \cH(t_j)
 V_{-}(k^2 n, t_j) 
 \cV_{+}(m, t_j)
$$
is, for any $\sigma_1, \sigma_2 >0$,
\begin{align*}
\cD_{-+}  = & \sum_{n\ge 1} \sum_{k \ge 1} \frac{ \tau(n, T )^2}{n^{1 + 2iT}  k^{1   + 2iT}}    
  \int_{-\infty}^{\infty}
W(t) \cH(t)
 V_{-}(k^2 n, t) 
 \cV_{+}(n, t)  \frac{d^* t}{2\pi^2}.
\end{align*}
By (\ref{V+}) and \eqref{H}, we have 
\begin{align*}
\cD_{-+}  \sim &
\sum_{n\ge 1} \sum_{k \ge 1} \frac{ \tau(n, T )^2}{n^{1 + 2iT}  k^{1   + 2iT}}    
  \int_{-\infty}^{\infty}
 \frac{W(t) V_{-}(t) }{|t|(4T^2-t^2)^{\frac{1}{2}}}
\frac{1}{2 \pi i} \int_{(\sigma_1)} e^{w^2 +i \frac{\pi}{2}w}  
\left( \frac{|t|(4T^2-t^2)^{\frac{1}{2}}}{4 \pi^2 k^2 n} \right)^{w}
 \frac{dw}{w} 
 \\
 & \times 
 \frac{1}{2 \pi i} \int_{(\sigma_2)} 8\pi  \tilde{h}(1-s) \left( \frac{|t|(4T^2-t^2)^{\frac{1}{2}}}{4 T} \right)^{s}  n^{-s} \frac{ds}{s} 
 \frac{d^* t}{2\pi^2}
 \\
 \sim  &  
  \int_{-\infty}^{\infty}
 \frac{W(t) V_{-}(t) }{|t|(4T^2-t^2)^{\frac{1}{2}}}
\frac{1}{2 \pi i} \int_{(\sigma_1)} e^{w^2 +i \frac{\pi}{2}w}  
\left( \frac{|t|(4T^2-t^2)^{\frac{1}{2}}}{4 \pi^2  } \right)^{w}
\sum_{k \ge 1} \frac{1}{ k^{1   + 2iT +2w}}
 \\
 &
 \times
 \frac{1}{2 \pi i} \int_{(\sigma_2)} 8\pi  \tilde{h}(1-s) \left( \frac{|t|(4T^2-t^2)^{\frac{1}{2}}}{4 T} \right)^{s}
 \sum_{n\ge 1}  \frac{ \tau(n, T )^2}{n^{1 + 2iT +w +s} } 
 \frac{ds}{s} 
 \frac{dw}{w} 
 \frac{d^* t}{2\pi^2}
 \\
 \sim   &  
  \int_{-\infty}^{\infty}
 \frac{W(t) V_{-}(t) }{|t|(4T^2-t^2)^{\frac{1}{2}}}
\frac{1}{2 \pi i} \int_{(\sigma_1)} e^{w^2 +i \frac{\pi}{2}w}   \frac{1}{2 \pi i} \int_{(\sigma_2)}
\left( \frac{|t|(4T^2-t^2)^{\frac{1}{2}}}{4   } \right)^{w+s}8\pi  \tilde{h}(1-s) 
\\
& \times
  \frac{ \zeta(1   + 2iT +2w) \zeta(1  +w +s) 
 \zeta(1 + 2iT +w +s)^2   
  \zeta(1 + 4iT +w +s)}{\zeta(2 + 4iT +2w +2s)} 
 \frac{ds}{s} 
 \frac{dw}{w} 
 \frac{d^* t}{2\pi^2}.
\end{align*}
Let
$\sigma_1= \frac{1}{101} \; < \; \sigma_2 = \frac{1}{100}$. We first move the $w$-integral to the left, to the line $\Re(w)=-\sigma_1$. In doing so, we cross simple poles at $w=0$ and $w= -iT$. By the same argument as given for $\cD_{-+}$ in the previous subsection, and keeping in mind that $|V_{-}(t)|=1$ by \eqref{V-size_1}, we have that the shifted integral and residue at $w= -iT$ contribute $O(T^{-\alpha})$ if $\alpha$ is chosen small enough. Thus we need only consider the contribution of the simple pole at $w=0$, which equals
\begin{align*}
&  
  \int_{-\infty}^{\infty}
 \frac{W(t) V_{-}(t) }{|t|(4T^2-t^2)^{\frac{1}{2}}}
   \frac{1}{2 \pi i} \int_{(\sigma_2)}
\left( \frac{|t|(4T^2-t^2)^{\frac{1}{2}}}{4 \pi^2  } \right)^{s}8\pi  \tilde{h}(1-s) 
\\
& \times
  \frac{ \zeta(1   + 2iT) \zeta(1  +s) 
 \zeta(1 + 2iT  +s)^2   
  \zeta(1 + 4iT +s)}{\zeta(2 + 4iT  +2s)} 
 \frac{ds}{s} 
 \frac{d^* t}{2\pi^2}.
\end{align*}
Next, we move the contour to the left to $\Re(s)=-\sigma_2 <0$, crossing a double pole at $s=0$, a double pole at $s=-2iT$ and a simple pole at $s=-4iT$. We can immediately say that the poles at at $s=-2iT$ and $s=-4iT$ contribute a negligible amount, by Lemma \ref{small-s}. Thus we only need to address the contributions of the double pole at $s=0$ and of the new integral at $\Re(s)=-\sigma_2 <0$.

For the pole at $s=0$ we follow the same type of calculations as we had for (\ref{eq:D_++}), but this time we have a double pole instead of a triple pole, so get that
\begin{align*}
 & \int_{-\infty}^{\infty}
 \frac{W(t) V_{-}(t) }{|t|(4T^2-t^2)^{\frac{1}{2}}} \zeta(1   + 2iT)\\
&\times \underset{s=0}{\mathrm{Res}}\left[
\left( \frac{|t|(4T^2-t^2)^{\frac{1}{2}}}{4 \pi^2  } \right)^{s}\frac{8\pi  \tilde{h}(1-s)}{s} 
  \frac{ \zeta(1  +s) 
 \zeta(1 + 2iT  +s)^2   
  \zeta(1 + 4iT +s)}{\zeta(2 + 4iT  +2s)} \right]
 \frac{d^* t}{2\pi^2}\\
&\ll \hat{h}(0)
|\zeta(1  + 4iT)| |\zeta(1  + 2iT)|^3 \log T  
\ll 
\hat{h}(0) 
|\zeta(1  + 2iT)|^3 (\log T)^{\frac{5}{3} + \epsilon}
\end{align*}
and this is smaller in size than the main term (\ref{eq:D_++}) that we saw for $\cD_{++}$.

It remains to consider the contribution of the shifted integral,
\[
\int_{-\infty}^\infty   \frac{W(t) |V_{-}(t)| }{|t|(4T^2-t^2)^{\frac{1}{2}}} |\mathcal{I R}(t)| d^*t \ll 
T^{-1} \int_{-\infty}^\infty  W(t)  |\mathcal{I R}(t)| dt,
\]
where, following the same type of notation as in the previous subsection, we denote
\begin{align*}
 \mathcal{I R}(t):= &
   \frac{1}{2 \pi i} \int_{(-\sigma_2)}
\left( \frac{|t|(4T^2-t^2)^{\frac{1}{2}}}{4 T  } \right)^{s}8\pi  \tilde{h}(1-s) 
\\
& \times
  \frac{ \zeta(1   + 2iT) \zeta(1  +s) 
 \zeta(1 + 2iT  +s)^2   
  \zeta(1 + 4iT +s)}{\zeta(2 + 4iT  +2s)} 
 \frac{ds}{s}. 
\end{align*}
The subconvexity estimate $\zeta(\frac{1}{2}+it) \ll |t|^{\frac{13}{84}+\epsilon}$ from \cite{B} (of sub-Weyl quality!) and the Phragm\'{e}n-Lindel\"{o}f principle together give that $\zeta(1- \sigma +it) \ll |t|^{\frac{13}{42}\sigma+ \epsilon}$, for $0 \le \sigma \le \frac{1}{2}$.
We  apply this estimate  to three (of the four) zeta factors in the numerator to get
\begin{align*}
|\mathcal{I R}(t)| &\ll   T^{\frac{39}{42}\sigma_2 + \epsilon} 
\int_{- T^{-\alpha}}^{T^\alpha}
    |  \tilde{h}(1-\sigma_2-iy) | |\zeta(1 -\sigma_2+iy) | \left( \frac{|t|(4T^2-t^2)^{\frac{1}{2}}}{4 T} \right)^{-\sigma_2}
  \frac{dy}{1+|y|} \\
  &\ll   T^{\frac{39}{42}\sigma_2 -(1-\alpha)\sigma_2 + \epsilon}\int_{- T^{-\alpha}}^{T^\alpha}
  \frac{|\zeta(1 -\sigma_2+iy) |}{1+|y|} dy\\
  &\ll   T^{\frac{39}{42}\sigma_2 -(1-\alpha)\sigma_2 + \epsilon} \ll T^{-\frac{1}{1400} + \frac{\alpha}{100}+\epsilon},
\end{align*}
so that
\[
T^{-1} \int_{-\infty}^\infty  W(t)  |\mathcal{I R}(t)| dt\ll T^{-1-\frac{1}{1400} + \frac{\alpha}{100}+\epsilon} \int_{-\infty}^\infty  W(t) dt\ll T^{-\frac{1}{1400} + \frac{\alpha}{100}+\epsilon},
\]
which is $O(T^{-\alpha})$ for $\alpha$ small enough.

\subsection{Diagonal $\cD_{+-}$}

The argument is the same as that given in the previous subsection for $\cD_{-+}$, and we get 
$$
\cD_{+-} \ll 
\hat{h}(0) 
|\zeta(1  + 2iT)|^3 (\log T)^{\frac{5}{3} + \epsilon}.
$$

\subsection{Diagonal $\cD_{--}$} 

The argument here is analogous to the one given previously for $\cD_{++}$ and as in that case, we get a main term.  
Using \eqref{H-}, we have
\begin{align*}
\cD_{--} &\sim 
    \int_{- \infty}^{\infty}     \frac{W(t) T^{2iT} \pi^{-4iT} e^{-2iT} }{|t|(4T^2-t^2)^{\frac{1}{2}}} 
  \frac{1}{2 \pi i} \int_{(\sigma_1)} e^{w^2 -i \frac{\pi}{2}w - i \pi s}  \frac{1}{2 \pi i} \int_{(\sigma_2)} 
\left( \frac{|t|(4T^2-t^2)^{\frac{1}{2}}}{4    } \right)^{w+s}  \pi^{-2w}T^{-s}
\\
 & \times  8\pi  \tilde{h}(1-s) 
  \frac{\zeta(1   +2iT +2 w) \zeta(1+w+s  )^2  \zeta(1+w+s +2iT) \zeta(1+w+s -2iT)}{\zeta(2+2w+2s )} 
  \frac{ds}{s} 
  \frac{dw}{w}     
    \frac{d^*t}{2  \pi^2}.
\end{align*}
As before, we let $\sigma_1= \frac{1}{101} \; < \; \sigma_2 = \frac{1}{100}$ and then shift the contours left. The main contribution arises from the simple pole at $w=0$ and triple pole at $s=0$. This is
\begin{align*}
\cD_{--}
& \sim 
\int_{- \infty}^{\infty}    \frac{W(t) T^{2iT} \pi^{-4iT} e^{-2iT} }{|t|(4T^2-t^2)^{\frac{1}{2}}}  4\pi \hat{h}(0) \frac{\zeta(1 +2iT)^2 \zeta(1 -2iT)}{\zeta(2)}  \log^2 \left(  \frac{|t|(4T^2-t^2)^{\frac{1}{2}}}{ T }\right)
  \frac{d^*t}{2  \pi^2}\\
  &  \sim  \hat{h}(0)
 \pi^{-4iT}  e^{-2iT}
 T^{2iT} 
 \frac{ 12 }{  \pi^2}  
 \zeta(1 +2iT )^2 
\zeta(1 -2iT ) \log^2 T,
  \end{align*}
  by the same final calculation leading to \eqref{eq:D_++}.

\subsection{The total diagonal contribution}

We collect all diagonal contributions from the sum (\ref{eq_final_expr_before_Kuznetsov}):
\begin{align}
\frac{\pi}{\zeta^2(1 -2i T) \zeta(1 +2i T)} 
& \left( \cD_{++} + \cD_{-+} + \bc \pi^{2iT}  \cD_{+-} + \bc \pi^{2iT} \cD_{--}  \right)  \nonumber \\
& \sim \frac{\pi}{\zeta^2(1 -2i T) \zeta(1 +2i T)} \left( \cD_{++}  
  + \bc \pi^{2iT} \cD_{--}  \right)    \nonumber \\
  & 
   \sim \hat{h}(0)  \frac{12}{\pi}    
    (\log T)^2
     \left( 1      
  +   \frac{\xi(1-2iT)}{\xi(1+2iT)} \pi^{-2iT}    e^{-2iT} T^{2iT}   
  \frac{\zeta(1 + 2iT)}{\zeta(1 - 2iT)}          \right)  \nonumber \\
  &
  \sim  \hat{h}(0)  \frac{12}{\pi}   
    (\log T)^2
     \left( 1      
  +   \frac{\Gamma(\frac{1}{2} -iT) }{ \Gamma(\frac{1}{2} +iT)  }    e^{-2iT} T^{2iT}   
          \right)  \nonumber \\
          &
          \sim  \hat{h}(0)  \frac{24}{\pi}    
    (\log T)^2,   \label{eq:total_diagonals}
\end{align}
where the last line follows 
by Stirling's approximation (\ref{eq:stirling}).

\section{The Eisenstein series contribution}

We recombine the Eisenstein series contributions from $\Xi_{++}$ and $\Xi_{-+}$ (while for $\Xi_{+-}$ and $\Xi_{--}$, one proceeds analogously) to get that
$\cE_{++} + \cE_{-+}$ is equal to
$$
-
\int_{- \infty}^{\infty}  \frac{W(t) \cH(t) }{|\zeta(1+ 2it)|^2}   
\left(
\sum_{n\ge 1} \sum_{k \ge 1} \frac{ \tau(n, T ) \tau(n, t)}{n^{\frac{1}{2} - iT}  k^{1   - 2iT}}  V_{+ }(k^2 n, t) 
+
\sum_{n\ge 1} \sum_{k \ge 1} \frac{ \tau(n, T ) \tau(n, t)}{n^{\frac{1}{2} + iT}  k^{1   + 2iT}} 
    V_{- }(k^2 n, t) \right) 
$$
$$
\times
   \sum_{m \ge 1}  \frac{   \tau(m, T) \tau(m, -t)}{ m^{\frac{1}{2} + i T }}
      \cV_{+}(m, t)
   \frac{d t}{2 \pi}.    
$$

The sum of double sums over $n$ and $k$ is an approximate functional equation (see e.g. \cite[Theorem 5.3]{IK}) of $|\zeta(\tfrac{1}{2} +it)|^2 \zeta(\tfrac{1}{2} -2iT +it) \zeta(\tfrac{1}{2} -2iT -it)$, and by Ramanujan's identity (\ref{eq:Ramanujan}),
$$
\sum_{m \ge 1}  \frac{   \tau(m, T) \tau(m, -t) }{ m^{\frac{1}{2} + i T +s}} =\frac{\zeta(\frac{1}{2}  -it +s ) \zeta(\frac{1}{2} +it +s ) 
\zeta(\frac{1}{2} +2iT -it +s ) \zeta(\frac{1}{2} +2iT +it +s )}{\zeta(1 +2iT +2s )}.
$$
 Therefore for $\sigma > \frac{1}{2}$, we have 
\begin{align}
& \frac{\pi}{\zeta^2(1 -2i T) \zeta(1 +2i T)} \left( \cE_{++} + \cE_{-+}  \right)   \nonumber \\
& =  
-\pi
\int_{- \infty}^{\infty} 
W(t) 
\frac{|\zeta(\tfrac{1}{2} +it)|^2 \zeta(\tfrac{1}{2} -2iT +it) \zeta(\tfrac{1}{2} -2iT -it) }{|\zeta(1+ 2it)|^2   \zeta^2(1 -2i T) \zeta(1 +2i T)}  
 \frac{1}{2 \pi i} \int_{(\sigma)}\cH(t)  \cH_{+}(s, t)  
\nonumber  \\
& \qquad \times  
\frac{\zeta(\frac{1}{2}  -it +s ) \zeta(\frac{1}{2} +it +s ) 
\zeta(\frac{1}{2} +2iT -it +s ) \zeta(\frac{1}{2} +2iT +it +s )}{\zeta(1 +2iT +2s )}
\frac{ds}{s}  
   \frac{d t}{2 \pi}.   \label{eq:Eisen_final}
\end{align}
Since $\sigma > \frac{1}{2}$, all the zeta factors in the numerator in the third line are absolutely bounded. So by \eqref{H}, the $s$-integral is $O(T^{ -2 +\sigma+ \epsilon})$ uniformly for any $t$ in the bulk range. The zeta values $\zeta(\tfrac{1}{2} -2iT +it)$, $\zeta(\tfrac{1}{2} -2iT -it)$ are each $O(T^{\frac{1}{6} + \epsilon})$ by Weyl's bound (although here, any subconvexity bound would suffice). Thus
\[
 \frac{\cE_{++} + \cE_{-+} }{\zeta^2(1 -2i T) \zeta(1 +2i T)} 
 \ll T^{ -2 +\sigma+\frac{1}{3}+\epsilon} \int_{-2T}^{2T} |\zeta(\tfrac{1}{2} +it)|^2 dt \ll T^{ -\frac23+\sigma+\epsilon},
\]
which is admissible if we choose  $\sigma=\frac{1}{2} + \frac{1}{100}$, say.

\section{The off-diagonal contribution} \label{off-diag}

We treat only the sum $\cO_{++}$ as the other three cases are similar. We will show that
\begin{align}
\label{off-diag-goal}  \cO_{++}  \ll T^{-1},
\end{align}
where
\begin{align}
 \cO_{++} = & \int_{-\infty}^{\infty} h(A) 
\sum_{n, m, k, c \ge 1}  \frac{\tau(n, T ) \tau(m, T)}{n^{\frac{1}{2} - iT} k^{1   - 2iT}  m^{\frac{1}{2} + i T }} \frac{S(n, m; c)}{c} 
\nonumber  \\
& \times
\int_{- \infty}^{\infty} \mathcal{J}\left( \frac{\sqrt{nm}}{c}, t \right)  
W(t) \cH(t)  
V_{+ }(k^2 n, t) \cV_{+}(m, t)
\frac{d^*t}{2 \pi} \ dA.   \label{eq:O++}
\end{align}
By the properties of the weight functions $V_{+ }(n k^2, t)$ and $\cV_{+}(m, t)$ given in Section \ref{Subsection_weights_support},  we may restrict the summation in (\ref{eq:O++}) to $nk^2 <T^{2+\epsilon}$ and $m< T^{1+\epsilon}$, up to an error of size $O(T^{-100})$. By a standard method  we may also restrict to $c < T^3$, see e.g. the paragraph at the top of page 1495 in \cite{BK}.  

Recall that by Lemma \ref{small-s}, we may assume that $|s|<T^{\alpha}$ in the definition \eqref{cV-def} of $\cV_{+}(m, t)$, up to negligible error. Also recall that in the definition \eqref{V+def} of $V_+(k^2 n,t)$ we may assume that $|w|<T^\epsilon$, by the rapid decay of $e^{w^2}$ in vertical lines.

Let $Z$ be a function as described in Lemma \ref{Lema:J_int} below, so that $Z(\frac{t}{T})$ is supported on $T^{1-2\alpha}<|t|<T^{1+2\alpha}$. Since $W(t)$ is already supported on $T^{1 - \alpha}\le |t| \le  2T - T^{1- \alpha}$ up to an error of $O(T^{-100})$, we may insert $Z(\frac{t}{T})$ into \eqref{eq:O++} up to $O(T^{-100})$ error. Thus to establish \eqref{off-diag-goal}, it suffices to show that for any $s,w\in\mathbb{C}$ with $\Re(s),\Re(w)>0$ fixed, $|w|<T^\epsilon$, $|s|<T^\alpha$, we have
\begin{multline}
\label{suffices} \sum_{\substack{
nk^2< T^{2+\epsilon} \\
m<T^{1+\epsilon}\\
c<T^3}}  \frac{\tau(n, T ) \tau(m, T)}{n^{\frac{1}{2} - iT+w} k^{1   - 2iT+2w}  m^{\frac{1}{2} + i T+s }} \frac{S(n, m; c)}{c} 
 \int_{- \infty}^{\infty} \mathcal{J}\left( \frac{\sqrt{nm}}{c}, t \right)  
Z\left(\frac{t}{T}\right) W(t) \cH(t)   \cH_{+}(s, t) \\
\times e^{w^2} \frac{\prod_{\pm} \Gamma_{\R}(\frac{1}{2}    +w \pm i t) \Gamma_{\R}(\frac{1}{2}-2iT +w \pm i t)}{\prod_{\pm} \Gamma_{\R}(\frac{1}{2}  \pm i t) \Gamma_{\R}(\frac{1}{2}-2iT  \pm i t)}   \frac{d^*t}{2 \pi} \ \ll T^{-\frac32}.
\end{multline}

Next we recall the following information about the $J$-Bessel function transform:
\begin{lem} \label{Lema:J_int} \cite[Lemma 3.3]{BK} Let $ 0 < \alpha < \frac{1}{100}$. For any $x > 0$ and any smooth, even function $Z$ compactly
supported on $(T^{-2\alpha}, T^{2\alpha}) \cup  (-T^{2\alpha}, -T^{-2\alpha})$ with derivatives satisfying $\| Z^{(k)} \|_{\infty} \ll_k (T^{2\alpha})^k$ for $k\ge 0$, we have
$$
\int_{- \infty}^{\infty} \frac{J_{2 i t}(2 \pi x)}{\cosh(\pi t )}  Z\left(\frac{t}{T} \right) t \, dt
$$
$$
=\frac{-i \sqrt{2}}{\pi} \frac{T^2}{\sqrt{x}} \Re\left( (1+i)e(x) \int_0^{\infty} t Z(t)  e\left( \frac{-t^2 T^2}{2 \pi^2 x} \right) dt  \right) +O\left( \frac{x}{T^{3-24 \alpha}} \right) +O(T^{-100}).
$$
The main term is $O(T^{-100})$ if $x < T^{2 -6 \alpha}$.
\end{lem}
\noindent First, we need to show that Lemma \ref{Lema:J_int} applies to  \eqref{suffices}, and for this we need to show that the functions
\[
 W(t), \  \ \  e^{w^2} \frac{\prod_{\pm} \Gamma_{\R}(\frac{1}{2}    +w \pm i t) \Gamma_{\R}(\frac{1}{2}-2iT +w \pm i t)}{\prod_{\pm} \Gamma_{\R}(\frac{1}{2}  \pm i t) \Gamma_{\R}(\frac{1}{2}-2iT  \pm i t)}, \ \ \ \cH(t)  \cH_{+}(s, t)
\]
can be `absorbed' into $Z(\frac{t}{T})$, which is to say that they satisfy the same bounds as
\[
\frac{d^k}{dt^k} Z\Big(\frac{t}{T}\Big) \ll_k T^{k(-1+2\alpha)}.
\]
For $W(t)$, this was already observed in \cite[Lemma 5.1]{BK}. For the remaining functions, this can easily be checked after using Stirling's expansion. The leading terms are, by \eqref{V+} and \eqref{H},
\[
e^{w^2} \frac{\prod_{\pm} \Gamma_{\R}(\frac{1}{2}    +w \pm i t) \Gamma_{\R}(\frac{1}{2}-2iT +w \pm i t)}{\prod_{\pm} \Gamma_{\R}(\frac{1}{2}  \pm i t) \Gamma_{\R}(\frac{1}{2}-2iT  \pm i t)}
= e^{w^2+i\frac{\pi}{2}w} \left( \frac{|t|(4T^2-t^2)^{\frac{1}{2}}}{4 \pi^2 } \right)^{w}+\ldots
\] 
and
\[
\cH(t)  \cH_{+}(s, t) = \frac{8\pi  \tilde{h}(1-s)}{|t|(4T^2-t^2)^{\frac{1}{2}}} \left( \frac{|t|(4T^2-t^2)^{\frac{1}{2}}}{4 T} \right)^{s} + \ldots
\]
for $t$ in the bulk range. Keep in mind that $|s|<T^\alpha$, which is crucial to control the size of derivatives with respect to $t$.

Since $\frac{\sqrt{nm}}{c} < T^{\frac{3}{2} + \epsilon}$ for all pairs $n, m$ in our summation \eqref{suffices}, we see that the main term from Lemma  \ref{Lema:J_int} gives a total contribution of $O(T^{-50})$ and the error term $O(\frac{\sqrt{nm}}{c T^{3-24 \alpha}})$ gives a contribution of
$$
O\left( \frac{T^{\epsilon}}{T^{2-\alpha}} \sum_{\substack{n k^2 < T^{2+\epsilon} \\ m< T^{ 1+ \epsilon} \\ c < T^3}}  \frac{|\tau(n, T ) \tau(m, T)|}{(nm)^{\frac{1}{2} } k} \frac{|S(n, m; c)|}{c}
\frac{\sqrt{nm}}{c T^{3-24 \alpha}}
 \right) =O(T^{-2 +25 \alpha +\epsilon}),
$$
using Weil's bound for Kloosterman sums. The bound is admissible since $\alpha < \frac{1}{100}$.

\section{Conclusion}

Taking everything together, 
by (\ref{eq:total_diagonals}) and the upper bounds for the Eisenstein series and off-diagonal contributions, we obtain the following proposition:

\begin{prop} \label{prop:cross_term} 
For $h$ as in Proposition \ref{the-prop},
as $T \rightarrow \infty$, we have
$$
\int_{-\infty}^\infty h(A) \ \Xi(A)  \ dA  \sim \hat{h}(0) \frac{24}{\pi}    
    (\log T)^2.
$$
\end{prop}

Now we are ready to prove the main Proposition. 
\begin{proof}[Proof of Proposition \ref{the-prop}] The proposition follows by combining Lemma \ref{lema:spectral_expansion} with (\ref{bound_E_A^2_continuous}), (\ref{bound_H_A_continuous}) and Propositions \ref{prop:constant}, \ref{prop:IMRN_main}, \ref{prop:H_A_H_A} and \ref{prop:cross_term}.
\end{proof}

\


\bibliographystyle{amsplain}

\end{document}